\documentclass[a4paper,12pt,draft]{article}
\usepackage{amssymb,amsmath,amsthm,latexsym}

\usepackage{color}

\usepackage[a4paper,margin=1in]{geometry} 
\usepackage{verbatim}
\usepackage{dsfont}
\usepackage{marginnote}

\theoremstyle{plain}
\newtheorem{thm}{Theorem}[section]
\newtheorem{prop}[thm]{Proposition}
\newtheorem{corollary}[thm]{Corollary}
\newtheorem{lemma}[thm]{Lemma}

\theoremstyle{definition}
\newtheorem{defn}[thm]{Definition}
\newtheorem{remark}[thm]{Remark}
\newtheorem{example}[thm]{Example}
\newtheorem*{ack}{Acknowledgement}
\numberwithin{equation}{section}

\newcommand{\real}{\mathds{R}}

\newcommand{\I}{\mathds{1}}
\newcommand\R{\mathds{R}}
\newcommand\E{\mathds{E}}
\newcommand\Pp{\mathds{P}}
\newcommand\N{\mathds{N}}

\newcommand\FF{\mathcal{F}}
\newcommand\EE{\mathcal{E}}
\newcommand\II{\mathcal{I}}

\newcommand\MM{\mathcal{M}}

\newcommand{\diam}{\operatorname{diam}}
\newcommand{\entropy}{\mathcal{H}}

\newcommand\newAF{\widetilde{A}}

\newcommand{\scalar}[1]{\langle #1\rangle}

\renewcommand{\le}{\leqslant}
\renewcommand{\leq}{\leqslant}
\renewcommand{\ge}{\geqslant}
\renewcommand{\geq}{\geqslant}

\begin{document}\allowdisplaybreaks

\title{\Large\bfseries Absolute continuity and singularity of probability measures induced by a  purely discontinuous Girsanov transform of a stable process}

\author{\itshape
    Ren\'e L.~Schilling
    \and
    \itshape Zoran Vondra\v{c}ek
}

\date{}

\maketitle

\begin{abstract}\noindent
    In this paper we study mutual absolute continuity and singularity of probability measures on the path space which are  induced by an isotropic stable L\'evy process and  the  purely discontinuous Girsanov transform  of this process.  We also look at the problem of finiteness of the relative entropy of these measures. An important tool in the paper is the question  under which circumstances the a.s.\ finiteness of an additive functional at infinity implies the finiteness of its expected value.


\end{abstract}

\section{Introduction}\label{sec-1}
One of the common issues when studying Markov processes is the question whether a random time-change of a conservative (transient) Markov process is again conservative. The time-change is usually realized as the generalized  inverse of a positive additive functional. Typical examples are functionals of the form $A_t=\int_0^t f(X_s)\, ds$ where $f$ is a real-valued function defined on the state space $E$ of the Markov process $X=(X_t,  \MM, \MM_t,  \Pp_x)$,  $t\ge 0$, $x\in E$. The time-changed process is again conservative if, and only, if $A_{\infty}=\infty$ a.s.; equivalently, it is not conservative if, and only, if $\Pp_x (A_{\infty}<\infty)>0$ for some $x$. The related question, whether the expectation of $\E_x A_{\infty}$ is finite, is relatively easy to answer. Note that $\E_x A_{\infty}= Gf (x)$ is the Green potential of $f$. Using potential-theoretic methods  we can reduce this problem to properties  of the corresponding Green function; this analytic problem is rather well understood. Hence, it is crucial to understand under which circumstances the finiteness of the additive functional at infinity,  $A_{\infty}<\infty$, will imply the finiteness of the expectation, $\E_x A_{\infty} < \infty$.

If $X$ is a diffusion on an interval in $\R$, additive functionals of the type $A_t=\int_0^{t}f(X_s)\, ds$ are sometimes called \emph{perpetual integral functionals}. The question of necessary and sufficient conditions for the a.s.~finiteness of $A_{\infty}$ has been addressed in \cite{KSY} following earlier work by various authors, notably by Engelbert and co-authors, see e.g.~\cite{ES}.

Additive functionals also appear in the study of absolute continuity of measures on the path space of a process.  In fact, by the Girsanov theorem, the Radon-Nikod\'ym density often takes the form $\exp(A_t)$ where $A_t$ is an additive functional as above. The problem when two probability measures are absolutely continuous over a finite time horizon is rather well understood. On the other hand, not much is known if the time-interval is infinite. Clearly, the question whether $A_{\infty}$ is finite or infinite plays a decisive role. For elliptic diffusions with drift this question has been addressed in \cite{B-AP}. To be more precise, let
$$
    \mathcal{L}_{a,b}
    =\frac12 \sum_{i,j=1}^d a_{ij}(x) \frac{\partial^2}{\partial x_i\partial x_j}+\sum_{i=1}^d b_i(x)\frac{\partial}{\partial x_i}
$$
be a diffusion operator on $\R^d$ with a $a(x) = (a_{ij}(x))_{i,j=1}^d$ positive definite and locally H\"older continuous diffusion coefficient and a locally H\"older continuous drift $b(x) = (b_i(x))_{i=1}^d$, and assume that the martingale problem for $(\mathcal{L}_{a,b},C_c^\infty(\R^d))$ is well posed. Denote by $\Pp_x^{a,b}$ the corresponding probability measure on the path space $C([0,\infty), \R^d)$, and the canonical process by $X=(X_t)_{t\ge 0}$. It is proved in \cite[Theorem 1]{B-AP} that for a further drift coefficient  $\hat{b}$ the following holds:
\begin{enumerate}
\item[(i)]
    $\Pp_x^{a,b}\perp \Pp_x^{a,\hat{b}}
    \iff
    \int_0^{\infty}\langle \hat{b}-b, a^{-1}(\hat{b}-b)\rangle (X_s)\, ds =\infty \ \ \text{a.s. } \Pp_x^{a,b} \text{  or a.s. } \Pp_x^{a,\hat{b}}$;
\item[(ii)]
    $\Pp_x^{a,\hat{b}}\ll \Pp_x^{a,b}
    \iff
    \int_0^{\infty}\langle \hat{b}-b, a^{-1}(\hat{b}-b)\rangle (X_s)\, ds < \infty \ \  \text{a.s. } \Pp_x^{a,\hat{b}}$.
\end{enumerate}
Ben-Ari and Pinsky also give a formula for the relative entropy $\entropy(\Pp_x^{a,b}; \Pp_x^{a,\hat{b}})$ of the measures $\Pp_x^{a,b}$ and $\Pp_x^{a,\hat{b}}$. Moreover, if both operators  $\mathcal{L}_{a,b}$ and $\mathcal{L}_{a,\hat{b}}$ are Fuchsian---roughly speaking: $a$ is uniformly elliptic and $(1+|x|)|b(x)|$, $(1+|x|)|\hat{b}(x)|$ are bounded---, then
$$
    \text{either}\quad
        \Pp_x^{a,b}\perp \Pp_x^{a,\hat{b}}
    \quad\text{or}\quad
    \Pp_x^{a,b}\sim \Pp_x^{a,\hat{b}};
$$
in the second case $\sup_{x\in \R^d}\entropy(\Pp_x^{a,b}; \Pp_x^{a,\hat{b}})<\infty$ and $\sup_{x\in \R^d}\entropy( \Pp_x^{a,\hat{b}}; \Pp_x^{a,b})<\infty$, cf.~\cite[Theorem 2]{B-AP}. Recall that the relative entropy of two measures $\nu$ and $\mu$ is defined as
$$
    \entropy(\nu; \mu)
    =
    \begin{cases}\displaystyle
        \int \frac{d\nu}{d\mu}\log \frac{d\nu}{d\mu}\, d\mu=\int\log \frac{d\nu}{d\mu} \, d\nu &\text{if\ \ } \nu\ll\mu,\\
         +\infty & \text{otherwise.}
    \end{cases}
$$
The first goal of this paper is to look into the question when the a.s.~finiteness of an additive functional implies the finiteness of its expectation; this problem is first considered in a rather general framework and then in the more specific framework of isotropic stable L\'evy processes in $\R^d$. Our second goal is to study the absolute continuity and mutual singularity of probability measures induced by a purely discontinuous Girsanov transform.
More precisely, let $X=(X_t,\MM,\MM_t,\Pp_x)$, $t\geq 0$, $x\in\R^d$, be a conservative symmetric (w.r.t.~Lebesgue measure) right Markov process in  $\R^d$  defined on the path space with filtration $(\MM_t)_{t\ge 0}$, and assume that $\MM=\sigma(\cup_{t\ge 0}\MM_t)$.
By $I_2(X)$ we denote all bounded, symmetric functions $F:\R^d\times \R^d\to\real$, vanishing on the diagonal, such that
$$
    \E_x\bigg[\sum_{s\le t}F^2(X_{s-},X_s)\bigg] < \infty
    \quad\text{for all\ }x\in\R^d,
$$
holds, see Definition \ref{d:i2x} in Section \ref{sec-2} for details. Such functions give rise to martingale additive functionals $(M^F_t)_{t\ge 0}$ whose quadratic variation is given by $[M^F]_t= \sum_{s\le t}F^2(X_{s-}, X_s)$. If $\inf_{x,y}F(x,y)>-1$, the solution to the SDE $L^F_t=1+\int_0^t L^F_{s-}\, dM^F_s$ is a positive local martingale, hence a positive supermartingale under each $\Pp_x$. By the general theory, there exists a family $\widetilde{\Pp}_x$ of (sub)-probability measures on $\MM$ such that ${d\widetilde{\Pp}_x} _{|\MM_t}=L^F_t \,{d\Pp_x} _{|\MM_t}$ for all $t\ge 0$.  Under these measures $X$ is a strong Markov process. We will write $\widetilde{X}=(\widetilde{X}_t, \MM, \MM_t, \widetilde{\Pp}_x)$ to denote this process. The process $\widetilde{X}$ is called a \emph{purely discontinuous Girsanov transform} of $X$, see \cite{CS03,S} and Section \ref{sec-2} for further details. Since $L^F_t>0$, we have ${d\widetilde{\Pp}_x} _{|\MM_t} \sim {d\Pp_x} _{|\MM_t}$ for all $t\ge 0$. We are interested under which conditions $\widetilde{\Pp}_x$ and $\Pp_x$ are mutually absolutely continuous on the whole time-interval $[0,\infty)$.  The following three theorems are the main results of the paper.

\begin{thm}\label{t:1}
    Let $X$ be a conservative symmetric right Markov process in $\R^d$.
    Assume that $F\in I_2(X)$ and $\inf_{x,y\in \R^d}F(x,y)>-1$.
    \begin{enumerate}
    \item[\upshape (a)]
        $\widetilde{\Pp}_x\perp \Pp_x$ if, and only, if $\sum_{t>0}F^2(X_{t-},X_t)=\infty$ $\Pp_x$ a.s.~or $\widetilde{\Pp}_x$ a.s.
    \item[\upshape (b)]
        $\widetilde{\Pp}_x\ll\Pp_x$ if, and only, if  $\sum_{t>0}F^2(X_{t-},X_t)<\infty$ $\widetilde{\Pp}_x$ a.s.
    \item[\upshape (\~b)]
        $\Pp_x\ll\widetilde{\Pp}_x$ if, and only, if $\sum_{t>0}F^2(X_{t-},X_t)<\infty$ $\Pp_x$ a.s.
    \item[\upshape (c)]
        $\entropy(\widetilde{\Pp}_x; \Pp_x)= \widetilde{\E}_x\left[ \sum_{t> 0}\left(\log(1+F)-\frac{F}{1+F}\right)(X_{t-},X_t)\right]$ and

        $\entropy(\widetilde{\Pp}_x; \Pp_x)<\infty$ if, and only, if $\widetilde{\E}_x \left[\sum_{t>0}F^2(X_{t-},X_t)\right]<\infty $.

    \item[\upshape (\~c)]
        $\entropy(\Pp_x; \widetilde{\Pp}_x)=\E_x\left[\sum_{t>0} \left(F-\log(1+F)\right)(X_{t-},X_t)\right]$ and

        $\entropy(\Pp_x; \widetilde{\Pp}_x)<\infty$ if, and only, if $\E_x \left[\sum_{t>0}F^2(X_{t-},X_t)\right]<\infty $.
    \end{enumerate}
\end{thm}

The proof of Theorem \ref{t:1} follows directly from our investigations in Section~\ref{sec-2} on purely discontinuous Girsanov transforms. Let us, for completeness, indicate already here how to derive the assertions from these results.
\begin{proof}[Proof of Theorem \ref{t:1}] Parts (a), (\~a) and (b), (\~b) are immediate consequences of Theorem \ref{t:main-1}. Parts (c), (\~c) follow from Proposition \ref{p:entropy} and Remark \ref{r:entropy} (a).
\end{proof}

In the next two theorems we
assume that $X$ is an isotropic $\alpha$-stable L\'evy  process.
\begin{thm}\label{t:2}
    Let $X$ be an isotropic $\alpha$-stable  L\'evy  process in $\R^d$. Assume  that $0 <\alpha< 2\wedge d$, $F\in I_2(X)$ and $\inf_{x,y\in \R^d}F(x,y) >-1$.
    \begin{enumerate}
    \item\label{t:2-a}
        Either $\widetilde{\Pp}_x\perp \Pp_x$ or $\widetilde{\Pp}_x\sim \Pp_x$.
    \item\label{t:2-b}
        If $\widetilde{\Pp}_x\sim\Pp_x$ and if there exist $C>0$ and $\beta>\alpha/2$ such that
        \begin{equation}\label{e:fuchsian-cond}
             0\leq  F(x,y)\le C\frac{|x-y|^{\beta}}{1+|x|^{\beta}+|y|^{\beta}}\quad \text{for all\ }x,y\in\R^d,
        \end{equation}
        then $\sup_{x\in \R^d}\entropy(\Pp_x;\widetilde{\Pp}_x )<\infty$.
    \end{enumerate}
\end{thm}

Note that contrary to \cite[Theorem 2]{B-AP}, we do not need a Fuchsian-type condition \eqref{e:fuchsian-cond} to conclude that dichotomy in part \ref{t:2-a} of Theorem \ref{t:2}  holds true. On the other hand, the next theorem shows that \eqref{e:fuchsian-cond} is needed for part \ref{t:2-b}.
\begin{thm}\label{t:3}
    Let $X$ be an isotropic $\alpha$-stable L\'evy  process in $\R^d$, $0<\alpha<2\wedge d$. For each $\gamma$ and $\beta$ satisfying $0<\gamma<\alpha/2 <\beta$  there exists some $F\in I_2(X)$ satisfying
    $$
        F(x,y)\le \frac12 \frac{|x-y|^{\beta}}{1+|x|^{\gamma}+|y|^{\gamma}}
    $$
    such that $\Pp_x\ll \widetilde{\Pp}_x$ and $\entropy(\Pp_x; \widetilde{\Pp}_x)=\infty$.
\end{thm}

These three theorems on purely discontinuous Girsanov transforms of isotropic stable L\'{e}vy processes are the analogues of Theorems 1--3 from \cite{B-AP}. The main ingredients used in the proof of Theorem \ref{t:1} are in the next section, while the proofs of Theorems \ref{t:2} and \ref{t:3} are given in the last section of this paper, building on the results obtained in earlier sections. In Section \ref{sec-2}, after recalling some necessary definitions and results from \cite{S}, we study absolute continuity and singularity of the measures $\Pp_x$ and $\widetilde{\Pp}_x$ for general strong Markov processes on $\R^d$. Section \ref{sec-3} is devoted to showing that the finiteness of the expectation of an additive functional $(A_t)_{t\ge 0}$ satisfying $A_{\infty}<\infty$ a.s.~is related to the lower boundedness of the function $u(x):=\E_x[e^{-A_{\infty}}]$. The motivation for this section comes from the need to understand the general principle underlying the first part of the proof of \cite[Theorem 2]{B-AP}. As an application, in Example \ref{ex:regular-diffusion} we give an alternative proof of (part of) \cite[Theorem 3]{KSY}. In Section \ref{sec-4} we look more closely in the case of an isotropic stable L\'evy process $X$. Following the ideas from \cite{BB00}, we first prove in Theorem \ref{t:hi} a Harnack inequality for $F$-harmonic functions of $X$---this can be thought of as a Harnack inequality for a Schr\"odinger-type semigroup of $X$. The main result of the section is Theorem \ref{t:finite-expectation} where we show that $\Pp_x$ a.s.~finiteness of $A_{\infty}=\sum_{t>0}F(X_{t-},X_t)$ implies finiteness of the expectation $\E_x A_{\infty}$ under an appropriate Fuchsian-type condition on the function $F$.

\begin{ack}
    The authors thank Jean Jacod for very helpful comments and for supplying the reference \cite[VIII (8.17), (8.18)]{Jac} which improved the presentation of Lemma \ref{l:stochastic-exp}. Part of this work was done while the second-named author was visiting TU Dresden.  Financial support through the Alexander-von-Humboldt foundation and the Croatian Science Foundation under the project no.~3526 is gratefully acknowledged.
\end{ack}

\section{Purely discontinuous Girsanov transforms:\\ Absolute continuity and singularity}\label{sec-2}
We begin with an auxiliary result which relates the finiteness of the quadratic variation of a local martingale with the convergence of the stochastic (Dol\'eans--Dade) exponential. Let $M=(M_t)_{t \ge 0}$, $M_0=0$, be a local martingale on a filtered probability space $(\Omega, \MM, (\MM_t),\Pp)$. As usual, we denote the quadratic variation of $M$ by $[M] = ([M])_{t\geq 0}$ and the predictable quadratic variation (or angle bracket) by $\scalar M = (\scalar M_t)_{t\geq 0}$; the angle bracket is the compensator of the quadratic variation, i.e.\ $[M]-\langle M\rangle$ is a local martingale. Further, $[M]_\infty = \sup_{t>0}[M]_t$, $\scalar M_\infty = \sup_{t>0}\scalar M_t$, and if the superscript `c' denotes the continuous part, then $\scalar{M^c} = [M^c] = [M]^c$ and
$$
    [M]_t
    =[M]^c_t +\sum_{s\le t}(\Delta M_s)^2
    =\scalar{M^c}_t +\sum_{s\le t}(\Delta M_s)^2.
$$
Let $\EE(M)=(\EE(M)_t)_{t\ge 0}$ be the stochastic exponential defined by
$$
    \EE(M)_t=\exp\left(M_t- \frac 12  [M]^c_t\right)\prod_{0<s\le t} (1+\Delta M_s)e^{ -\Delta M_s}.
$$
The stochastic exponential is the unique solution of the SDE $L_t=1+\int_0^t L_{s-}\, dM_s$, hence a local martingale. For these facts we refer to \cite{HWY} or \cite{Pr}.

We will need the following simple observation: Let $(a_n)_{n\ge 1}$ be a sequence of real numbers such that $a_n>-1$. Then
\begin{equation}\label{e:simple-lemma}
    \sum_{n\ge 1}a_n^2 <\infty
    \iff
    \sum_{n\ge 1}\left(\frac{a_n}{1+a_n}\right)^2<\infty
    \iff
    \prod_{n\ge 1}\frac{1+a_n}{\left(1+\frac12 a_n\right)^2}>0.
\end{equation}

\begin{lemma}\label{l:stochastic-exp}
    Let $M$ be a local martingale such that $\Delta M_t >-1$ for all $t>0$.
    \begin{enumerate}
    \item\label{l:stochastic-exp-a} $\{[M]_{\infty}=\infty\}\subset \{\lim_{t\to \infty} \EE(M)_t=0\}$ a.s.

    \item\label{l:stochastic-exp-b} If there exists some $C>0$ such that $\Delta M_t < C$ for all $t>0$, then

    $\{\lim_{t\to \infty} \EE(M)_t=0\} = \{[M]_{\infty}=\infty\}$ a.s.
    \end{enumerate}
\end{lemma}

\begin{remark}
    The a.s.\ equality $\left\{\lim_{t\to \infty} \EE(M)_t=0\right\}=\left\{\langle M \rangle_{\infty}=\infty\right\}$ is well known for \emph{continuous} local martingales, cf.~\cite[Exercise IV-3.25]{RY}; Lemma \ref{l:stochastic-exp} extends this to martingales with jumps.

    A proof of Lemma \ref{l:stochastic-exp}\ref{l:stochastic-exp-b} can be easily deduced from \cite[Corollaire VIII (8.17)]{Jac}, while part \ref{l:stochastic-exp-a} is explicitly stated in \cite[Corollaire VIII (8.18)]{Jac}.  Below we give an elementary proof of part \ref{l:stochastic-exp-a} and prove part \ref{l:stochastic-exp-b} under the additional assumption that $\Delta M_t <1$ for all $t>0$.
\end{remark}

\begin{proof}[Proof of Lemma \ref{l:stochastic-exp}]
\ref{l:stochastic-exp-a} Consider the local martingale $\frac 12 M$. Then
$$
    \EE\left(\tfrac{1}{2}M\right)_t=\exp\left(\tfrac12 M_t-\tfrac18 [M]^c_t\right)\prod_{s\le t} \left(1+\tfrac 12 \Delta M_s\right)e^{-\frac12 \Delta M_s}.
$$
Since $\Delta\big(\frac 12M\big)_t = \frac 12\Delta M_t > -\frac 12$, $\EE\big(\frac 12M\big)$ is a positive local martingale, hence a supermartingale, and therefore the $\lim_{t\to \infty}\EE\big(\frac 12M\big)_t$ exists a.s. Furthermore,
\begin{align*}
    \EE\left(\tfrac{1}{2}M\right)_t^2
    &= \exp\left(M_t-\tfrac14 [M]^c_t\right)\prod_{s\le t} \left(1+\tfrac 12 \Delta M_s\right)^2 e^{- \Delta M_s}\\
    &= \EE(M)_t \exp\left(\tfrac14 [M]^c_t\right)\prod_{s\le t} \frac{\left(1+\frac 12\Delta M_s\right)^2}{1+\Delta M_s},
\end{align*}
which can be rearranged as
\begin{equation}\label{e:stochastic-exp-1}
    \EE(M)_t
    =\EE\left(\tfrac 12 M\right)_t^2 \exp\left(-\tfrac14 [M]^c_t\right)\prod_{s\le t} \frac{1+\Delta M_s}{\left(1+\frac 12\Delta M_s\right)^2}.
\end{equation}
Since $[M]_t=[M]^c_t +\sum_{s\le t}(\Delta M_s)^2$, we find
$$
    \big\{[M]_{\infty}=\infty\big\}
    \subset \big\{[M]^c_{\infty}=\infty\big\} \cup \Big\{\sum\nolimits_{t > 0}(\Delta M_t)^2=\infty\Big\}.
$$
Note that $0<(1+x)/(1+ \frac 12 x )^2 \le 1$ for all $x\ge -1$ implying that the product appearing in \eqref{e:stochastic-exp-1} stays bounded between $0$ and $1$; moreover, $\EE(\frac 12M)^2_t$ is bounded since it converges as $t\to\infty$. Now there are two possibilities:
\begin{enumerate}
\item[(i)] $[M]^c_{\infty}=\infty$, then the right-hand side of \eqref{e:stochastic-exp-1} tends to zero;

\item[(ii)] $\sum_{t\ge 0} (\Delta M_t)^2 =\infty$, then by \eqref{e:simple-lemma},
    $
        \prod_{t\ge 0} \frac{1+\Delta M_t}{\left(1+\frac 12 \Delta M_t\right)^2}=0
    $,
and the right-hand side of \eqref{e:stochastic-exp-1} tends to zero;
\end{enumerate}
in both cases, $\lim_{t\to \infty}\EE(M)_t=0$.

\medskip\noindent
\ref{l:stochastic-exp-b} Assume that $\Delta M_t<1$ for all $t>0$ and $\lim_{t\to \infty}\EE(M)_t=0$. From $\EE(X)\EE(Y)=\EE(X+Y+[X,Y])$ we get that $\EE(M)\EE(-M)=\EE(-[M])$. Because of $\Delta M_t <1$, $\EE(-M)$ is a nonnegative local martingale, hence a nonnegative supermartingale which is a.s.\ convergent as $t\to \infty$.  Thus, on the set $\{\lim_{t\to \infty}\EE(M)_t=0\}$ we see that the product $\EE(M)_t\EE(-M)_t$ converges to $0$ a.s., hence, $\lim_{t\to \infty}\EE(-[M])_t=0$ a.s. From
\begin{align*}
    \EE(-[M])_t
    &= \exp(-[M]_t)\prod_{s\le t}(1-\Delta [M]_s)e^{\Delta [M]_s}\\
    &= \exp\bigg(-[M]^c_t-\sum_{s\le t}\Delta[M]_s\bigg)\prod_{s\le t}(1-\Delta [M]_s)e^{\Delta [M]_s}\\
    &= \exp (-[M]^c_t)\prod_{s\le t} (1-(\Delta M_s)^2)
\end{align*}
we conclude that $[M]^c_{\infty}=\infty$ or $\prod_{t> 0} (1-(\Delta M_t)^2)=0$ holds true. The latter is equivalent to $\sum_{t> 0}(\Delta M_t)^2=\infty$. Thus we have $[M]_{\infty}=[M]^c_{\infty}+\sum_{t\ge 0}(\Delta M_t)^2=\infty$ on the set $\{\lim_{t\to \infty}\EE(M)_t=0\}$.
\end{proof}

A version of the following lemma is stated and proved in \cite[Lemma 1]{B-AP}.
\begin{lemma}\label{l:from-B-AP}
Let $(\Omega, \MM)$ be a measurable space and $(\MM_t)_{t\ge 0}$ a filtration such that $\MM=\sigma\left(\bigcup_{t\ge 0} \MM_t\right)$. Let $\Pp$ and $\widetilde{\Pp}$ be probability measures on $(\Omega, \MM)$ such that $\widetilde{\Pp}_{|\MM_t}\ll \Pp_{|\MM_t}$ for all $t\ge 0$. Then
\begin{enumerate}
\item\label{l:from-B-AP-a}
    $\widetilde{\Pp}\ll \Pp
    \iff \widetilde\Pp\Big(\limsup\limits_{t\to \infty} \frac{d\widetilde{\Pp}_{|\MM_t}}{d\Pp_{|\MM_t}}<\infty\Big) = 1$.

\item\label{l:from-B-AP-b}
    $\widetilde{\Pp}\perp \Pp
    \iff \widetilde\Pp\Big(\limsup\limits_{t\to \infty} \frac{d\widetilde{\Pp}_{|\MM_t}}{d\Pp_{|\MM_t}}=\infty\Big)=1
    \iff \Pp\Big(\limsup\limits_{t\to \infty} \frac{d\widetilde{\Pp}_{|\MM_t}}{d\Pp_{|\MM_t}}=0\Big)=1$.

\item\label{l:from-B-AP-c}
    $\widetilde{\Pp}\ll \Pp
    \implies \entropy(\widetilde{\Pp}; \Pp)=\lim_{t\to \infty} \entropy(\widetilde{\Pp}_{|\MM_t}; \Pp_{|\MM_t})$.
\end{enumerate}
\end{lemma}

In the remainder of this section we adopt the setting of \cite{S} with the simplification that the state space is $\R^d$ and the process has infinite lifetime: Let $X=(\Omega, \MM, \MM_t, \theta_t, X_t, \Pp_x)$ be a symmetric (w.r.t.\ Lebesgue measure) right Markov process on $\R^d$ with infinite lifetime. We will always work with the canonical representation of $X$, i.e.\ $\Omega=D([0,\infty),\R^d)$ is the Skorokhod space of all c\`adl\`ag functions $\omega:[0,\infty)\to\R^d$, $X_t$ is the coordinate projection $X_t(\omega)=\omega(t)$ and $\MM=\sigma\left(\cup_{t\ge 0}\MM_t\right)$. Under $\Pp_x$, $X$ is a strong Markov process with initial condition $X_0=x$. The shift operators $\theta_t$, $t\ge 0$, satisfy $X_s\circ \theta_t=X_{s+t}$ for all $t,s\ge 0$. By $(N, H)$ we denote the  L\'evy system of $X$. This means that $H=(H_t)_{t\ge 0}$ is a positive continuous additive functional of $X$ with bounded $1$-potential and $N(x,dy)$ is a kernel from $(\R^d, \mathcal B(\R^d))$ to $(\R^d, \mathcal B(\R^d))$ satisfying $N(x,\{x\})=0$ for all $x\in \R^d$ and
$$
    \E_x\bigg[\sum_{s\le t}f(X_{s-}, X_s)\bigg]
    =\E_x\left[\int_0^t \int_{\R^d}f(X_{s-}, y)N(X_{s-}, dy)\, dH_S\right],\quad x\in\R^d,
$$
for any non-negative Borel function $f$ on $\R^d\times \R^d$ vanishing on the diagonal.

\begin{defn}[\cite{CS03,S}]\label{d:i2x}
    \begin{enumerate}
    \item\label{d:i2x-a}
    The class $J(X)$ consists of all bounded, symmetric functions $F:\R^d\times\R^d\to\R$ which vanish on the diagonal and satisfy
    $$
        \lim_{t\to 0} \sup_{x\in \R^d} \E_x\left[\int_0^t\int_{\R^d} |F(X_{s-},y)| N(X_{s-},dy)\,  dH_s\right] = 0.
    $$

    \item\label{d:i2x-b}
    The class $I_2(X)$ consists of all bounded, symmetric functions $F:\R^d\times\R^d\to\R$ which vanish on the diagonal and satisfy for all $x\in\R^d$ and $t>0$
    $$
        \E_x\bigg[\sum_{s\le t}F^2(X_{s-},X_s)\bigg]
        =\E_x \left[\int_0^t\!\!\! \int_{\R^d} F^2(X_{s-},y) N(X_{s-},dy)\,  dH_s\right] <\infty.
    $$
    \end{enumerate}
\end{defn}

\begin{remark}\label{r:25}
(a) Since
$$
    \E_x \left[\int_0^t\!\!\! \int_{\R^d} F^2(X_{s-},y) N(X_{s-},dy)\, dH_s \right]
    \le \|F\|_{\infty} \E_x \left[\int_0^t\!\!\! \int_{\R^d} |F(X_{s-},y)| N(X_{s-},y)\, dH_s \right],
$$
we see, with a simple application of the Markov property, that $J(X)\subset I_2(X)$.

\medskip\noindent (b)
    Since the integrator $s\mapsto H_s$ is continuous and $s\mapsto X_s$ is c\`adl\`ag and has at most countably many discontinuities, we may replace in the integrals appearing in Definition \ref{d:i2x} $X_{s-}$ by $X_s$.

\medskip\noindent (c)
    If $\inf_{x,y}F(x,y)>-1$, then $F\in I_2(X)$ implies $\log(1+F)\in I_2(X)$.
\end{remark}

Let $F\in I_2(X)$. For $n\in \N$ define
$$
    M^{F,n}_t
    :=\sum_{s\le t}F(X_{s-},X_s)\I_{\{|X_s-X_{s-}|>\frac{1}{n}\}}
    - \int_0^t\!\!\! \int_{\R^d} F(X_{s-},y)\I_{\{|y-X_{s-}|>\frac{1}{n}\}} N(X_{s-},dy)\, dH_s.
$$
Then $M^{F,n}$ is a pure jump martingale additive functional (MAF) of $X$. Note that $\Delta M^{F,n}_s =F(X_{s-},X_s)$ and the quadratic variations are given by $[M^{F,n}]^c_t=0$ and
\begin{align*}
    [M^{F,n}]_t
    &= \sum_{s\le t}F^2(X_{s-}, X_s) \I_{\{|X_s-X_{s-}|>\frac{1}{n}\}},\\
    \langle M^{F,n}\rangle_t
    &=  \int_0^t\!\!\! \int_{\R^d} F^2(X_{s-},y) \I_{\{|y-X_{s-}|>\frac{1}{n}\}} N(X_{s-},dy)\, dH_s,
\end{align*}
According to \cite[p.~493]{S} the $L^2(\Pp_x)$-
limit $M^F_t=\lim_{n\to \infty}M^{F,n}_t$ exists and $M^F$ is a MAF of $X$. In particular, $M^F_t$ and
\begin{align}
    [M^F]_t
    &=  \sum_{s\le t}F^2(X_{s-}, X_s),\label{e:qvM}\\
    \langle M^F\rangle_t
    &=  \int_0^t\!\!\! \int_{\R^d} F^2(X_{s-},y)  N(X_{s-},dy)\, dH_s. \notag
\end{align}
are defined in $L^2$. The condition $\inf_{x,y} F(x,y)>-1$ ensures that the limit $M^F = \lim_{n\to\infty}M^{F,n}$ has again jumps bounded strictly from below by $-1$, i.e.\ $\Delta M_t^F > -1$. If $F\in J(X)$,  we could directly define $M^F$ by the a.s.\ expression $\sum_{s\le t}F(X_{s-},X_s) -\int_0^t\!\! \int_{\R^d}F(X_{s-},y)N(X_{s-},dy)\, dH_s$.

\medskip
Assume that $-c:=\inf_{x,y}F(x,y)>-1$. Then $-c\le F(x,y)\le C$ for some $C>0$. Set $L^{F,n}_t:=\EE(M^{F,n})_t$ and $L^F_t:=\EE(M^F)_t$. Clearly,
\begin{align*}
    L^{F,n}_t
    &=  \exp\left(M^{F,n}_t\right)\prod_{s\le t} (1+F(X_{s-},X_s))\I_{\{|X_s-X_{s-}|>\frac{1}{n}\}} \exp\left(-F(X_{s-},X_s)\I_{\{|X_s-X_{s-}|>\frac{1}{n}\}}\right)\\
    &= \exp\bigg(M^{F,n}_t+\sum_{s\le t} \left(\log(1+F)-F\right)(X_{s-},X_s)\I_{\{|X_s-X_{s-}|>\frac{1}{n}\}} \bigg)\\
    &= \exp\bigg( \sum_{s\le t} \log(1+F(X_{s-},X_s)\I_{\{|X_s-X_{s-}|>\frac{1}{n}\}})\\ &\qquad\qquad\mbox{}-\int_0^t\int_{\R^d}F(X_{s-},y)\I_{\{|y-X_{s-}|>\frac{1}{n}\}}\, N(X_{s-},dy)\, dH_s\bigg),
\end{align*}
and
\begin{align}
    L^F_t
    &=  \exp\left(M^F_t\right)\prod_{s\le t} (1+F(X_{s-},X_s)) \exp\left(-F(X_{s-},X_s)\right)\notag\\
    &=  \exp\Bigg(M^F_t+\sum_{s\le t} \big(\log(1+F)-F\big)(X_{s-},X_s)\Bigg).\label{e:LF}
\end{align}
Since $|\log(1+F)-F|\le c_1 F^2$ for some $c_1>0$, we see that $L^F_t\in (0,\infty)$. Because of $M_t^F \in L^2(\Pp_x)$ and $F\in I_2(X)$ we see that $\log L_t^F\in L^1(\Pp_x)$. It is proved in \cite[p.~494]{S} that $L^{F,n}_t$ converges to $L^F_t$ in probability as $n\to \infty$.

Recall that $L_F$ is under each $\Pp_x$ a non-negative local martingale, hence a supermartingale. By \cite[Section 62]{Sh} there exists a family $(\widetilde{\Pp}_x)_{x\in \R^d}$ of (sub-)probability measures on $\MM$ such that
$$
    {d\widetilde{\Pp}_x} _{|\MM_t}=L^F_t \,{d\Pp_x} _{|\MM_t}
    \quad\text{for all $t\ge 0$};
$$
under these measures $X$ is a right process which we denote by $\widetilde{X}=(\widetilde{X}_t, \MM, \MM_t, \widetilde{\Pp}_x)$. The process $\widetilde{X}$ is called a \emph{purely discontinuous Girsanov transform} of $X$. Since $L^F_t>0$ we see that ${d\widetilde{\Pp}_x} _{|\MM_t} \sim {d\Pp_x} _{|\MM_t}$ for all $t\ge 0$. We are interested when $\widetilde{\Pp}_x\sim \Pp_x$ or $\widetilde{\Pp}_x \perp \Pp_x$.

We need the following result from \cite[Proposition 2.3]{S}.
\begin{prop}\label{p:song}
    Assume that $f:\R^d\times \R^d\to [0,\infty)$ is a measurable function vanishing on the diagonal. Then we have for all $t\ge 0$ and $x\in \R^d$
    \begin{equation}\label{e:new-levy}
        \widetilde{\E}_x \bigg[\sum_{s\le t} f(X_{s-}, X_s)\bigg]
        =\widetilde{\E}_x\left[\int_0^t\!\!\! \int_{\R^d} f(X_{s-},y)(1+F(X_{s-},y))N(X_{s-},dy) \, dH_s\right].
    \end{equation}
    In particular, $((1+F(x,y))N(x,dy), H_s)$ is a L\'evy system for $\widetilde{X}$.
\end{prop}

Set $F_1:=-\frac{F}{1+F}$. From $-1<-c\le F(x,y)\le C$ we see that $-\frac{C}{1+C}\le F_1(x,y)\le \frac{c}{1-c}$. Hence, $F_1$ is symmetric, bounded and $\inf_{x,y}F_1(x,y)>-1$. By \cite[p.~497]{S}, we have that $F_1\in I_2(\widetilde{X})$.
Define
\begin{align*}
    &\widetilde{M}^{F_1, n}_t\\
    &:= \sum_{s\le t}F_1(\widetilde{X}_{s-}, \widetilde{X}_s)\I_{\{|\widetilde{X}_s - \widetilde{X}_{s-}|>\frac{1}{n}\}} -\int_0^t\!\!\! \int_{\R^d} F_1(1+F)(\widetilde{X}_{s-},y) \I_{\{|y-\widetilde{X}_{s-}|>\frac{1}{n}\}} N(\widetilde{X}_{s-},dy)\, dH_s\\
    &=  \sum_{s\le t}F_1(\widetilde{X}_{s-}, \widetilde{X}_s)\I_{\{|\widetilde{X}_{s-}- \widetilde{X}_s|>\frac{1}{n}\}} + \int_0^t\!\!\! \int_{\R^d} F(\widetilde{X}_{s-},y)\I_{\{|y-\widetilde{X}_{s-}|>\frac{1}{n}\}} N(\widetilde{X}_{s-},dy)\, dH_s.
\end{align*}
Then $\widetilde{M}^{F_1, n}$ is a MAF of $\widetilde{X}$ and by the same argument as before it converges 
to $\widetilde{M}^{F_1}$ which is again a MAF of $\widetilde{X}$. Note that
\begin{align}
    [\widetilde{M}^{F_1}]_t
    &=  \sum_{s\le t}F_1^2(\widetilde{X}_{s-}, \widetilde{X}_s),\label{e:qv-tildeM}\\
    \langle \widetilde{M}^{F_1}\rangle_t
    &=  \int_0^t\!\!\! \int_{\R^d} F_1^2(\widetilde{X}_{s-},y)  \big(1+F(\widetilde{X}_{s-},y)\big)   N(\widetilde{X}_{s-},dy)\,  dH_s.\notag
\end{align}
Let $\widetilde{L}^{F_1,n}$ and $\widetilde{L}^{F_1}$, be the solutions to the SDEs $\widetilde{L}^{F_1,n}_t=1+\int_0^t \widetilde{L}^{F_1,n}_{s-}\, d\widetilde{M}^{F_1, n}_s$ and $\widetilde{L}^{F_1}_t=1+\int_0^t \widetilde{L}^{F_1}_{s-}\, d\widetilde{M}^{F_1}_s$, respectively. From \cite[p.~497]{S} we know that
\begin{equation}\label{e:inverse-density}
    \widetilde{L}^{F_1,n}_t=\frac{1}{L^{F,n}_t} 
\quad\text{and}\quad
    \widetilde{L}^{F_1}_t=\frac{1}{L^F_t}\, \quad \widetilde{\Pp}_x\text{~a.s.}
\end{equation}
As before we see that $\widetilde{L}^{F_1}_t\in L^1(\widetilde{\Pp}_x)$.

\begin{remark}
    Since ${\widetilde{\Pp}_x}{_{|\MM_t}} \sim {\Pp_x}{_{|\MM_t}}$, we also have $\widetilde{L}^{F_1}_t= 1/L^F_t$ $\Pp_x$ a.s.
\end{remark}

\begin{thm}\label{t:main-1}
Assume that $F\in I_2(X)$ and $\inf_{x,y}F(x,y)>-1$. Then
    \begin{align*}
        \widetilde{\Pp}_x \ll \Pp_x      &\iff \widetilde{\Pp}_x\left(\sum_{t> 0} F^2(X_{t-},X_t)<\infty\right)=1\\
        \Pp_x \ll \widetilde{\Pp}_x      &\iff \Pp_x\left(\sum_{t>0} F^2(X_{t-},X_t)<\infty\right)=1\\
        \widetilde{\Pp}_x \perp \Pp_x    &\iff \widetilde{\Pp}_x\left(\sum_{t> 0} F^2(X_{t-},X_t)=\infty\right)=1\\
        \Pp_x \perp \widetilde{\Pp}_x    &\iff \Pp_x\left(\sum_{t> 0} F^2(X_{t-},X_t)=\infty\right)=1
    \end{align*}
\end{thm}
Since $\widetilde{\Pp}_x \perp \Pp_x$ if, and only if, $\Pp_x \perp \widetilde{\Pp}_x$, Theorem \ref{t:main-1} immediately entails the following zero--two-law.
\begin{corollary}[Zero--two-law]
    Assume that $F\in I_2(X)$ and $\inf_{x,y}F(x,y)>-1$. Then
    $\widetilde{\Pp}_x\left(\sum_{t> 0} F^2(X_{t-},X_t)=\infty\right)+\Pp_x\left(\sum_{t> 0} F^2(X_{t-},X_t)=\infty\right) = 0$ or $2$ according to $\widetilde{\Pp}_x \sim \Pp_x$ or $\widetilde{\Pp}_x \bot \Pp_x$.
\end{corollary}
\begin{proof}[Proof of Theorem~\ref{t:main-1}]
Note that
$$
    \frac{{d\widetilde{\Pp}_x}_{|\MM_t}}{{d\Pp_x}_{|\MM_t}}
    =L^F_t=\frac{1}{\widetilde{L}^{F_1}_t}=\frac{1}{\EE(\widetilde{M}^{F_1})_t}
    \quad \Pp_{x|\MM_t}\text{~a.s.}
$$
Since $\Pp_{x|\MM_t}\sim \widetilde\Pp_{x|\MM_t}$ and since the densities are $\MM_t$ measurable, the above equality holds a.s.\ for both $\Pp_x$ and $\widetilde\Pp_x$. Hence,
$$
    \limsup_{t\to \infty}\frac{{d\widetilde{\Pp}_x}_{|\MM_t}}{{d\Pp_x}_{|\MM_t}}=\infty
    \iff
    \lim_{t\to \infty}\EE(\widetilde{M}^{F_1})_t=0
    \quad\widetilde{\Pp}_x\text{~a.s.}
$$
Since $-1<\inf_{x,y}F(x,y)\le \sup_{x,y}F(x,y)<\infty$, we
get that $-1<\inf_{x,y} F_1(x,y)\le \sup_{x,y} F_1(x,y)<\infty$. From Lemma \ref{l:stochastic-exp} we conclude that
$$
    \left\{\limsup_{t\to \infty}\frac{{d\widetilde{\Pp}_x}_{|\MM_t}}{{d\Pp_x}_{|\MM_t}}=\infty\right\}
    =
    \left\{[\widetilde{M}^{F_1}]_{\infty}=\infty\right\}
    \quad\widetilde{\Pp}_x\text{~a.s.}
$$
By Lemma \ref{l:from-B-AP}
\begin{align*}
    \widetilde{\Pp}_x \ll \Pp_x      &\iff [\widetilde{M}^{F_1}]_{\infty}<\infty \quad \widetilde{\Pp}_x\text{~a.s.}\\
    \widetilde{\Pp}_x \perp \Pp_x    &\iff [\widetilde{M}^{F_1}]_{\infty}=\infty \quad \widetilde{\Pp}_x\text{~a.s.}
\end{align*}
From \eqref{e:qv-tildeM} we conclude that
$$
    [\widetilde{M}^{F_1}]_{\infty}
    =\sum_{t > 0}F_1^2(X_{t-},X_t)
    =\sum_{t > 0} \left(\frac{F(X_{t-},X_t)}{1+F(X_{t-},X_t)}\right)^2.
$$
Now the first and the third equivalence follow from \eqref{e:simple-lemma}.

\medskip\noindent
The second and the fourth equivalence are proved analogously.  We start with the identity
$$
    \frac{{d\Pp_x}_{|\MM_t}}{{d\widetilde{\Pp}_x}_{|\MM_t}}=\frac{1}{L^F_t}=\frac{1}{\EE(M^F)_t} \quad \widetilde\Pp_{x|\MM_t}\text{~a.s.},
$$
and conclude, as before, that
\begin{align*}
    \Pp_x \ll \widetilde{\Pp}_x      &\iff [M^F]_{\infty}<\infty \quad \Pp_x\text{~a.s.}\\
    \Pp_x \perp \widetilde{\Pp}_x    &\iff [M^F]_{\infty}=\infty \quad \Pp_x\text{~a.s.}
\end{align*}
The claim now follows from \eqref{e:qvM}.
\end{proof}

In the next proposition we compute the relative entropies $\entropy(\Pp_x;\widetilde{\Pp}_x)$ and $\entropy(\widetilde{\Pp}_x; \Pp_x)$.
\begin{prop}\label{p:entropy}
Let $F\in I_2(X)$ and $\inf_{x,y}F(x,y)>-1$.
\begin{enumerate}
\item\label{p:entropy-a}
    Assume that $\widetilde{\Pp}_x \ll \Pp_x$. Then
    \begin{align}
        \entropy(\widetilde{\Pp}_x; \Pp_x)
        &=  \widetilde{\E}_x\left[\sum_{t> 0} \left(F_1-\log(1+F_1)\right)(X_{t-},X_t)\right] \label{e:entropy-tPP}\\
        &=  \widetilde{\E}_x\left[ \sum_{t> 0}\left(\log(1+F)-\frac{F}{1+F}\right)(X_{t-},X_t)\right]; \notag
    \end{align}
    in particular, $\entropy(\widetilde{\Pp}_x; \Pp_x)<\infty$ if, and only, if $\widetilde{\E}_x \left[\sum_{t> 0}F^2(X_{t-},X_t)\right]<\infty$.

\item\label{p:entropy-b}
    Assume that $\Pp_x\ll \widetilde{\Pp}_x$. Then
    \begin{equation}\label{e:entropy-PtP}
    \entropy(\Pp_x; \widetilde{\Pp}_x)=\E_x\left[\sum_{t>0} \left(F-\log(1+F)\right)(X_{t-},X_t)\right],
    \end{equation}
    and $\entropy(\Pp_x; \widetilde{\Pp}_x)<\infty$ if, and only, if $\E_x \left[\sum_{t> 0}F^2(X_{t-},X_t)\right]<\infty$.
\end{enumerate}
\end{prop}
\begin{proof} We begin with part \ref{p:entropy-b}. By the definition of the entropy,
    \begin{gather*}
        \entropy\left({\Pp_x}_{|\MM_t}; {\widetilde{\Pp}_x}{_{|\MM_t}}\right)
        =  \E_x\left[\log\frac{{d\Pp_x}_{{|\MM_t}}}{{d\widetilde{\Pp}_x}_{|\MM_t}}\right]
        =  \E_x\left[\log\frac{1}{L^F_t}\right]
        = -\E_x\big[\log L^F_t\big].
    \end{gather*}
    Combining this with \eqref{e:LF} yields
    \begin{align*}
        \entropy\left({\Pp_x}_{|\MM_t}; {\widetilde{\Pp}_x}{_{|\MM_t}}\right)
        &=  -\E_x[M^F_t]+\E_x\left[\sum_{s\le t}\left(F-\log(1+F)\right)(X_{s-},X_s)\right]\\
        &=  \E_x\left[\sum_{s\le t}\left(F-\log(1+F)\right)(X_{s-},X_s)\right]
    \end{align*}
    since $M^F_t$ is a martingale. As $\Pp_x\ll \widetilde{\Pp}_x$ we get with Lemma \ref{l:from-B-AP} (c) that
    $$
        \entropy(\Pp_x; \widetilde{\Pp}_x)
        =\lim_{t\to \infty} \entropy\left({\Pp_x}{_{|\MM_t}}; {\widetilde{\Pp}_x}{_{|\MM_t}}\right)
        =\E_x\left[\sum_{t> 0} \left(F-\log(1+F)\right)(X_{t-},X_t)\right].
    $$
    Note that $-1<-c\le F\le C$. Hence, there are two constants $c_1,c_2>0$ such that $c_1 F^2 \le F-\log(1+F)\le c_2 F^2$. This proves the second part of the claim.

\medskip\noindent
    The identity \eqref{e:entropy-tPP} of part \ref{p:entropy-a} follows, using the analogue of \eqref{e:LF}, in the same way as \eqref{e:entropy-PtP}, while the second equality is clear. As before we can show that $\entropy(\widetilde{\Pp}_x; \Pp_x)<\infty$ if, and only, if  $\widetilde{\E}_x \left[\sum_{t\ge 0}F_1^2(X_{t-},X_t)\right]<\infty$. Since $c_3 F^2\le F^2/(1+F)^2 = F_1^2 \le c_4 F^2$ for some constants $c_3, c_4>0$, the claim follows.
\end{proof}

\begin{remark}\label{r:entropy}
   (a) Assume that $-1<\inf_{x,y}F(x,y)\le \sup_{x,y}F(x,y)<\infty$. Then the conclusion of Proposition \ref{p:entropy}\ref{p:entropy-b} holds regardless of $\Pp_x\ll \widetilde{\Pp}_x$: If $\Pp_x$ is not absolutely continuous with respect to $\widetilde{\Pp}_x$ then, by definition, $\entropy(\Pp_x; \widetilde{\Pp}_x)=\infty$. Moreover, by Theorem \ref{t:main-1}, $\Pp_x\left(\sum_{t>0}F^2(X_{t-}, X_t)=\infty\right)>0$ implying that $\E_x\left[\sum_{t>0} F^2 (X_{t-},X_t)\right]=\infty$, hence the right-hand side of \eqref{e:entropy-PtP} is infinite as well. A similar argument applies to part \ref{p:entropy-a} of the proposition.

   \medskip\noindent
   (b) Assume that the L\'evy system $(N,H)$ satisfies $H_s\equiv s$. If $\Pp_x\ll \widetilde{\Pp}_x$, we can rewrite the entropy $\entropy(\Pp_x; \widetilde{\Pp}_x)$  in the following form:
    \begin{align*}
        \entropy\left({\Pp_x}_{|\MM_t}; {\widetilde{\Pp}_x}{_{|\MM_t}}\right)
        &=  \E_x\left[\sum_{s\le t}\left(F-\log(1+F)\right)(X_{s-},X_s)\right]\\
        &=  \E_x\left[\int_0^t\!\!\! \int_{\R^d} \left(F-\log(1+F)\right)(X_{s-},y) N(X_{s-},dy)\, ds\right],
    \end{align*}
    hence, by Lemma \ref{l:from-B-AP}~\ref{l:from-B-AP-c},
    \begin{align*}
        \entropy(\Pp_x; \widetilde{\Pp}_x)
        &=  \E_x\left[\int_0^{\infty}\!\!\! \int_{\R^d} \left(F-\log(1+F)\right)(X_{s-},y) N(X_{s-},dy)\, ds\right]\\
        &=  \E_x\left[\int_0^{\infty} h(X_s)\, ds\right]\\
        &= Gh(x)
    \end{align*}
    where
    $$
        h(z):=\int_{\R^d}(F-\log(1+F))(z,y)N(z,dy),
    $$
    and $G$ denotes the potential (Green) operator of $X$. If $\Pp_x$ is not absolutely continuous with respect to $\widetilde{\Pp}_x$, then by part (a), both $\entropy(\Pp_x; \widetilde{\Pp}_x)$ and $Gh(x)$ are infinite.

    Similarly, using Proposition \ref{p:song}, we have
    \begin{align*}
        \entropy(\widetilde{\Pp}_x; \Pp_x)
        &=  \widetilde{\E}_x\left[\int_0^{\infty}\left(\log(1+F)-\frac{F}{1+F}\right)(1+F)(X_{s-},y)N(X_{s-},dy)\, ds\right]\\
        &=  \widetilde{\E}_x\left[\int_0^{\infty}\big((1+F)\log(1+F)-F\big)(X_{s-},y) N(X_{s-},dy)\, ds\right]\\
        &=  \widetilde{\E}_x \left[\int_0^{\infty}h_1(X_s)\, ds \right]\\
        &=  \widetilde{G}h_1(x),
    \end{align*}
    where
    $$
        h_1(z):=\int_{\R^d}\big( (1+F)\log(1+F)-F\big)(z,y)\,N(z,dy),
    $$
    and $\widetilde{G}$ denotes the potential (Green) operator of $\widetilde{X}$.
\end{remark}

In the next corollary we assume that for every $x\in \R^d$, there exists some $t>0$ (which may even depend on $x$) such that $X_t$ has a strictly positive density under both $\Pp_x$ and $\widetilde\Pp_x$, that is $\Pp_x(X_t\in C)=\int_C p(t,x,z)\, dz$ with $p(t,x,z)>0$ and $\widetilde{\Pp}_x(X_t\in C)=\int_C \tilde{p}(t,x,z)\, dz$ with $\tilde p(t,x,z)>0$.

\begin{corollary}\label{c:one-implies-all}
Let $F\in I_2(X)$ and $\inf_{x,y}F(x,y)>-1$.
\begin{enumerate}
\item\label{c:one-implies-all-a}
    Assume that for every $x\in \R^d$ there is some $t$ such that $X_t$ has a strictly positive transition density under $\widetilde{\Pp}_x$. If $\widetilde{\Pp}_x\ll \Pp_x$ (resp.\ $\widetilde{\Pp}_x\perp \Pp_x$) for some $x\in \R^d$, then this is true for all $x\in \R^d$.

\item\label{c:one-implies-all-b}
    Assume that
    for every $x\in\R^d$ there is some $t$ such that $X_t$ has a strictly positive transition density under $\Pp_x$. If $\Pp_x\ll \widetilde{\Pp}_x$ (resp.\ $\Pp_x\perp \widetilde{\Pp}_x$) for some $x\in \R^d$, then this is true for all $x\in \R^d$.
\end{enumerate}
\end{corollary}
\begin{proof}
\ref{c:one-implies-all-b}
Assume that $\Pp_x\perp \widetilde{\Pp}_x$, let $f(y)=\Pp_y \big(\sum_{s> 0}F^2(X_{s-},X_s)=\infty\big)$ and pick $t$ as in the statement of the corollary. Since
$$
    \sum_{s> 0}F^2(X_{s-},X_s) = \sum_{s\le t}F^2(X_{s-},X_s) +\sum_{s>t}F^2(X_{s-},X_s)
$$
and the first sum is always finite (as $F\in I_2(X)$), we see that
\begin{align*}
    f(x)
    =  \Pp_x\left(\sum_{s> 0}F^2(X_{s-},X_s)=\infty\right)
    &=\Pp_x\left(\sum_{s>t}F^2(X_{s-},X_s)=\infty\right)\\
    &=  \Pp_x\left(\sum_{s> 0}F^2(X_{s-}\circ\theta_t, X_s\circ\theta_t)=\infty\right)\\
    &=  \E_x\left[\Pp_{X_t}\left(\sum_{s> 0}F^2(X_{s-},X_s)=\infty\right)\right]\\
    &= \E_x f(X_t)\\
    &=  \int_{\R^d}f(z)p(t,x,z) \, dz.
\end{align*}
By Theorem \ref{t:main-1} we have $f(x)=1$, and this implies that $f(z)=1$ for Lebesgue a.e.~$z\in \R^d$. Then, however, $f(y)=\int_{\R^d}f(z)p(t,y,z)\, dz=1$ for every $y\in \R^d$.

Assume now that $\Pp_x\ll \widetilde{\Pp}_x$, and let  $g(y)=\Pp_{y}\big(\sum_{s> 0}F^2(X_{s-},X_s)<\infty\big)$, so that $g(x)=1$. The same argument as above shows that $g(y)=1$ for all $y\in \R^d$.

\medskip\noindent
Part \ref{c:one-implies-all-a} is proved in the same way as \ref{c:one-implies-all-b}.
\end{proof}

Recall that the invariant $\sigma$-field $\II$ is defined as
$$
    \II=\big\{\Lambda\in \MM \,:\, \theta_t^{-1} \Lambda =\Lambda \ \text{for all } t\ge 0\big\}.
$$
\begin{corollary}\label{c:dichotomy}
    Assume that $F\in I_2(X)$ and $\inf_{x,y}F(x,y)>-1$. Fix $x\in \R^d$. If the invariant $\sigma$-field $\II$ is trivial under both $\Pp_x$ and $\widetilde{\Pp}_x$, then either $\widetilde{\Pp}_x\sim \Pp_x$ or $\widetilde{\Pp}_x\perp  \Pp_x$.
\end{corollary}
\begin{proof}
    Pick $\Lambda=\left\{\sum_{s> 0} F^2(X_{s-}, X_s)=\infty\right\}$. Then $\Lambda \in \II$, hence by the assumption $\Pp_x(\Lambda)=0$ or $1$, and $\widetilde{\Pp}_x(\Lambda)=0$ or $1$.

    If $\Pp_x(\Lambda)=1$, then by Theorem \ref{t:main-1} we first have $\Pp_x \perp \widetilde{\Pp}_x$, and then it follows that $\widetilde{\Pp}_x(\Lambda)=1$. The rest of the proof follows by exchanging $\Pp_x$ and $\widetilde{\Pp}_x$.
\end{proof}

Assume that for all $x\in \R^d$, $X_t$ admits a positive transition density $p(t,x,z)$ under $\Pp_x$. Let $\Lambda \in \II$ and define $\phi(y):=\Pp_y(\Lambda)$. Then,  by the Markov property,
$$
    \phi(x)
    =\E_x [\I_{\Lambda}]
    =\E_x [\I_{\Lambda}\circ \theta_t]
    =\E_x\left[\Pp_{X_t} \Lambda\right]
    =\E_x[\phi(X_t)]
    =\int_{\R^d} \phi(z)p(t,x,z)\, dz.
$$
If $\phi(x)=1$, we get from $0\le \phi\le 1$ that $\phi(z)=1$ Lebesgue a.e.; in the same way as in the proof of Corollary \ref{c:one-implies-all} it follows that $\phi\equiv 1$. Similarly, if $\phi(x)=0$, then $\phi\equiv 0$. In particular, this shows that \emph{if $\II$ is trivial under $\Pp_x$ for some $x\in \R^d$, then it is trivial under $\Pp_x$ for all $x\in \R^d$}. Under appropriate conditions, the same conclusion holds for $\widetilde{\Pp}_x$.

\section{Finiteness of the expectation of additive functionals}\label{sec-3}
Let $X=(\Omega, \MM, \MM_t, \theta_t, X_t, \Pp_x)$ be a strong Markov process with state space $\R^d$, assume that $\MM = \sigma\big(\bigcup_{t\geq 0}\MM_t\big)$, and let $A=(A_t)_{t\ge 0}$ be a \emph{perfect} additive functional of $X$, cf.~\cite{BG}. This means that $A$ is non-negative, adapted and there exists a set $\Lambda\in \MM$ such that $\Pp_x(\Lambda)=1$ for all $x\in \R^d$, such that on $\Lambda$
\begin{itemize}
\item
    $t\mapsto A_t$ is non-decreasing, right-continuous and $A_0=0$,
\item
    $A_{t+s}=A_s+A_t\circ \theta_s$ for all $s,t\ge 0$.
\end{itemize}
We will be mainly interested in the following two types of perfect additive functionals:
\begin{enumerate}
\item[(i)]
    $A_t:=\int_0^t f( X_s )\, ds$ where $f:\R^d\to [0,\infty)$ is a measurable function;
\item[(ii)]
    $A_t:=\sum_{s\le t} F(X_{s-}, X_s)$ where $F:\R^d\times \R^d \to [0,\infty)$ is a measurable function vanishing on the diagonal.
\end{enumerate}
In this section we discuss the problem when $\Pp_x(A_{\infty}<\infty)=1$ implies $\E_x [A_{\infty}]<\infty$.

Set $M_t:=e^{-A_t}$ for $t\ge 0$. Then $M=(M_t)_{t\ge 0}$ is a perfect multiplicative functional of $X$ taking values in $[0,1]$, and $M_{\infty}=\lim_{t\to \infty}M_t=e^{-A_{\infty}}$ is also well defined. Clearly, $M_{\infty}>0$ if, and only, if $A_{\infty}<\infty$. We also note that $M_0=1$. From the defining property of a multiplicative functional, $M_{t+s}=M_s\cdot (M_t\circ \theta_s)$ we get, letting $t\to \infty$, that $M_{\infty}=M_s \cdot (M_{\infty} \circ \theta_s)$, $\Pp_x$ a.s. for all $x\in \R^d$; the exceptional set does not depend on $x$.

Define $u:\R^d \to [0,1]$ by $u(x):=\E_x[M_{\infty}]$. Then $u$ is measurable, and by \cite[Theorem 2.5]{CS03}  $(u(X_t))_{t\ge 0}$ is a c\`adl\`ag process. We will prove several lemmas involving this function $u$.

\begin{lemma}\label{l:u-1}
    The process $\big(u(X_t)M_t\big)_{t\ge 0}$ is a bounded $\Pp_x$-martingale for all $x\in \R^d$. In particular
    \begin{equation}\label{e:u-1}
        u(x)=\E_x[u(X_t)M_t].
    \end{equation}
\end{lemma}
\begin{proof}
    By the Markov property we have
    \begin{equation}\label{e:1a}\begin{aligned}
        \E_x[M_{\infty}\, |\, \MM_s]
        &= \E_x[M_s \cdot (M_{\infty} \circ \theta_s)\, |\, \MM_s] \\
        &= M_s\E_x[M_{\infty} \circ \theta_s\, |\, \MM_s] \\
        &= M_s \E_{X_s}[M_{\infty}]\\
        &=M_s u(X_s).
    \end{aligned}\end{equation}
    The expression on the right-hand side is obviously bounded. Since the left-hand side is a martingale, the claim follows.
\end{proof}

From now on we will assume that $\Pp_x(A_{\infty}<\infty)=1$ for every $x\in \R^d$. This implies that $\Pp_x(M_{\infty}>0)=1$ for every $x\in \R^d$, hence $u>0$.

\begin{lemma}\label{l:u-2}
    It holds that $\Pp_x\left(\lim_{t\to \infty} u(X_t) = 1\right)=1$  for every $x\in \R^d$.
\end{lemma}
\begin{proof}
    Since $\big(u(X_t)M_t\big)_{t\ge 0}$ is a bounded $\Pp_x$-martingale, the martingale convergence theorem shows that the limit $\lim_{t\to \infty} u(X_t)M_t$ exists $\Pp_x$ a.s.

    On the other hand, $\lim_{t\to \infty}M_t=M_{\infty}>0$ $\Pp_x$ a.s. Hence, $\lim_{t\to \infty} u(X_t)$ exists $\Pp_x$ a.s.  Letting $t\to \infty$ in \eqref{e:u-1} we get from the bounded convergence theorem that
    $$
        \E_x\left[\lim_{t\to \infty}u(X_t) M_{\infty}\right]=u(x)=\E_x[M_{\infty}].
    $$
    Since $0\le u\le 1$ and $\Pp_x(M_{\infty}>0)=1$, the claim follows.
\end{proof}

For a measurable function $F:\R^d\times \R^d \to [0,\infty)$ vanishing on the diagonal, let $\widetilde{F}(x,y):=1-e^{-F(x,y)}$ and $\newAF_t:=\sum_{s\le t}\widetilde{F}(X_{s-},X_s)$.
\begin{lemma}\label{l:u-3}
\begin{enumerate}
\item\label{l:u-3-a}
    Let $A_t=\int_0^{\infty}f(X_s)\, dA_s$. For all $t>0$ it holds that
    \begin{align}
        u(x)
        &=\E_x\left[u(X_t)-\int_0^t u(X_s)\, dA_s\right]\label{e:u-3a-c} \\
        &= 1-\E_x\left[\int_0^{\infty} u(X_s)\, dA_s\right]\label{e:u-3b-c}.
    \end{align}
\item\label{l:u-3-b}
    Let $A_t=\sum_{s\le t}F(X_{s-},X_s)$  and $\newAF_t:=\sum_{s\le t} \big(1-e^{-F(X_{s-},X_s)}\big)$.  For all $t>0$ it holds that
    \begin{align}
        u(x)
        &=\E_x\left[u(X_t)-\int_0^t u(X_s)\, d\newAF_s\right]\label{e:u-3a-d} \\
        &= 1-\E_x\left[\int_0^{\infty} u(X_s)\, d\newAF_s\right]\label{e:u-3b-d}.
    \end{align}
\end{enumerate}
\end{lemma}
\begin{proof}
    We only prove part \ref{l:u-3-b} of the lemma, part \ref{l:u-3-a} follows in a similar way.

    We begin with  \eqref{e:u-3a-d}:  Set $\widetilde{F}(x,y) = 1-e^{-F(x,y)}$.  Because of \eqref{e:1a} we have
    \begin{align*}
        \E_x&\left[u(X_t)-\sum_{s\le t} u(X_s)\widetilde{F}(X_{s-},X_s) \right]\\
        &= \E_x\left[M_t^{-1}\E_x(M_{\infty} \mid \MM_t)-\sum_{s\le t} M_s^{-1}\E_x(M_{\infty} \mid \MM_s)\, \widetilde{F}(X_{s-},X_s)\right]\\
        &= \E_x\left[M_t^{-1}M_{\infty}\right]-\sum_{s\le t} \E_x\left[\E_x\left(M_s^{-1} M_{\infty} \widetilde{F}(X_{s-},X_s) \bigm| \MM_s\right)\right]\\
        &= \E_x\left[M_t^{-1} M_{\infty}\right]-\E_x\left[M_{\infty}\sum_{s\le t} M_s^{-1}  \widetilde{F}(X_{s-},X_s)\right]\\
        &= \E_x\left[M_{\infty}\,e^{A_t} \right]-\E_x\left[M_{\infty}\sum_{s\le t} e^{A_s}(1-e^{-F(X_{s-},X_s)})\right]\\
        &\stackrel{(*)}{=} \E_x\left[M_{\infty}\,e^{A_t} \right]-\E_x\left[M_{\infty}\big(e^{A_t}-1\big)\right]\\
        &= \E_x[M_{\infty}]\\
        &=u(x).
    \end{align*}
     In the equality marked by (*) we used the fact that $e^{A_t}$ is of purely discontinuous type, see \cite[top of p.~381]{Sh}; (in the corresponding calculation for continuous additive functionals we can use here $\int_0^t e^{A_s}\,dA_s = e^{A_t}-1$).

    Letting $t\to \infty$ in \eqref{e:u-3a-d} and a combination of Lemma \ref{l:u-2} with the  bounded and monotone convergence theorems gives \eqref{e:u-3b-d}.
\end{proof}

As a direct consequence of Lemma \ref{l:u-3} we see that the process $N=(N_t)_{t\ge 0}$, defined by $N_t:=u(X_t)-\int_0^t u(X_s)\, dA_s$, respectively $N_t:=u(X_t)-\int_0^t u(X_s)\, d\newAF_s$, is a $\Pp_x$-martingale for all $x\in \R^d$.

For a measurable set $D\subset \R^d$, denote by $\tau_D=\inf\{t>0:\, X_t\notin D\}$ the first exit time of $X$ from $D$. Since $u(X_s)M_s =\E[M_{\infty} \mid \MM_s]$ is a bounded martingale, cf.\ Lemma \ref{l:u-1}, the next lemma follows from optional stopping.
\begin{lemma}\label{l:u-4}
    For every open set $D\subset \R^d$ it holds that
    \begin{equation}\label{e:u-4}
        u(x)
        =\E_x\left[u(X_{\tau_D}) M_{\tau_D}\right]
        \quad \text{for all\ }x\in D.\footnotemark[1]
    \end{equation}
\end{lemma}
\footnotetext[1]{Because of Lemma~\ref{l:u-2} the right-hand side is also defined if $\tau_D=+\infty$. In this case we have $u(X_{\tau_D}) = 1$ on the set $\{\tau_D=+\infty\}$.}

\begin{prop}\label{p:u-6}
\begin{enumerate}
\item\label{p:u-6-a}
    Let $A_t=\int_0^t f(X_s)\, ds$. If $\E_x[A_{\infty}]\le c<\infty$, then $u(x)\ge e^{-c}>0$.
    Conversely, the condition $\inf_{x\in \R^d} u(x)=c>0$ implies $\sup_{x\in \R^d} \E_x[A_{\infty}]\le c^{-1}-1$.
    Thus, $\inf_{x\in \R^d} u(x)>0$ if, and only, if $\sup_{x\in \R^d} \E_x[A_{\infty}]<\infty$.
\item\label{p:u-6-b}
    Let $A_t=\sum_{s\le t}F(X_{s-},X_s)$ and $\newAF_t:=\sum_{s\le t} \big(1-e^{-F(X_{s-},X_s)}\big)$. If $\E_x[A_{\infty}]\le c<\infty$, then $u(x)\ge e^{-c}>0$.
     Conversely,  the condition $\inf_{x\in \R^d} u(x)=c>0$ implies $\sup_{x\in \R^d} \E_x[\newAF_{\infty}]\le c^{-1}-1$.
     If, in addition,  $0\le F(x,y)\le C$ for all $x,y\in \R^d$, then $\inf_{x\in \R^d} u(x)=c>0$ implies $\sup_{x\in \R^d} \E_x[A_{\infty}]<\infty$.
\end{enumerate}
\end{prop}
\begin{proof}
    Again we  only  prove part \ref{p:u-6-b}, part \ref{p:u-6-a} being similar.

    Suppose that $\E_x[A_{\infty}]\le c<\infty$. Jensen's inequality for the convex function $e^{-t}$ yields
    $$
        0<e^{-c}\le e^{-\E_x[A_{\infty}]}\le \E_x\left[e^{-A_{\infty}}\right]=\E_x[M_{\infty}]=u(x).
    $$
    Conversely, assume that $\inf_{x\in \R^d} u(x)=c>0$. It follows from \eqref{e:u-3b-d} that
    $$
        1-c
        \ge \E_x\left[\int_0^{\infty}u(X_s)\, d\newAF_s\right]
        \ge c \, \E_x[\newAF_{\infty}],
    $$
    implying that  $\sup_{x\in \R^d} \E_x[\newAF_{\infty}]\le c^{-1}-1$.

    Assume now that $0\le F(x,y)\le C$ for all $x,y\in \R^d$. There exists some constant $\kappa=\kappa(C)\ge 1$ such that $\kappa\lambda \le 1-e^{-\lambda}\leq\lambda$ for all $\lambda\in [0,C]$. This implies that $\kappa F(x,y)\le \widetilde{F}(x,y) \leq F(x,y)$ for all $x,y\in \R^d$, hence $\kappa A_t\le \newAF_t \leq A_t$ for all $t\in [0,\infty]$. Therefore, $\sup_{x\in \R^d} \E_x[A_{\infty}]\le \kappa^{-1} \sup_{x\in \R^d} \E_x[\newAF_{\infty}] <\infty$.
\end{proof}

\begin{prop}\label{p:u-7}
    Assume that $X$ is a strong Feller process and  $\lim_{t\to 0} \sup_{x\in \R^d} \E_x [A_t] =0$.
    \begin{enumerate}
    \item\label{p:u-7-a}
        Let $A_t=\int_0^{\infty}f(X_s)\, ds$. Then $u(x) = \E_x[M_\infty]$ is continuous.
    \item\label{p:u-7-b}
        Let $A_t=\sum_{s\le t}F(X_{s-},X_s)$ with $F \geq 0$ bounded. Then $u(x) = \E_x[M_\infty]$ is continuous.
    \end{enumerate}
\end{prop}
\begin{proof} We prove \ref{p:u-7-b},  as part \ref{p:u-7-a} is similar.
    From \eqref{e:u-3a-d} we know that
    $$
        u(x)-\E_x[u(X_t)]=\E_x \left[\int_0^t u(X_s)\, d\newAF_s\right].
    $$
    By the strong Feller property, $x\mapsto \E_x[u(X_t)]$ is continuous. On the other hand, by using notation from the proof of Proposition \ref{p:u-6}\ref{p:u-6-b},
    $$
        \E_x\left[\int_0^t u(X_s)\, d\newAF_s\right]
        \le \E_x \left[\int_0^t  d\newAF_s\right]
        =\E_x[\newAF_t]
        \le  \frac 1\kappa\, \E_x[A_t],
    $$
    which converges, by assumption, uniformly to zero as $t\to 0$. Thus $u$ is the uniform limit of the continuous functions $x\mapsto \E_x[u(X_t)]$ and, therefore, itself continuous.
\end{proof}

\begin{example}\label{ex:BM-radial}
    Let $X=(X_t,\Pp_x)$ be a Brownian motion in $\R^d$, $d\ge 3$. Assume that $f:\R^d \to [0,\infty)$ is bounded and radial, i.e.\ $f(x)=f(y)$ if $|x|=|y|$. Let $A_t:=\int_0^t f(X_s)\, ds$, set $M_t:=e^{-A_t}$, assume that $A_{\infty}<\infty$ $\Pp_x$~a.s., and let $u(x):=\E_x[M_{\infty}]$. Then:

    \medskip
    $u$ is radial: Let $|x|=|y|$ and assume that $U$ is a rotation around the origin such that $Ux=y$.  Since $f$ is radial, we have $\int_0^{\infty}f(U X_s)\, ds=\int_0^{\infty} f(X_s)\, ds$. This implies that the distributions of $A_{\infty}$ under $\Pp_y$ and $\Pp_x$ coincide.

    \medskip
    $u$ is continuous: This follows from Proposition \ref{p:u-7}\ref{p:u-7-a}.

    \medskip
    $u$ is bounded from below: It is enough to prove that $\liminf_{|x|\to \infty} u(x)>0$. Otherwise there would exist a sequence $(r_n)_{n\ge 1}$ such that $\lim_{n\to \infty}r_n=\infty$ and $u(x)\le 2^{-n}$ for $|x|=r_n$. Since $X$ has a.s.~continuous paths and $\limsup_{t\to \infty}|X_t|=+\infty$ a.s.,
    there exists for every $n\ge 1$ some $t_n=t_n(\omega)$ such that $|X_{t_n}(\omega)|=r_n$, hence $u(X_{t_n}(\omega))\le 2^{-n}$. This, however, contradicts the fact that $\lim_{t\to \infty}u(X_t)=1$ a.s.

    \medskip\noindent
    It follows from Proposition \ref{p:u-6}\ref{p:u-6-a} that $\sup_{x\in \R^d} \E_x A_{\infty} <\infty$.
\end{example}

\begin{example}\label{ex:regular-diffusion}
    Suppose that $X=(X_t,\Pp_x)$ is a strong Markov process on $\R$ such that $\lim_{t\to \infty}X_t=+\infty$. Assume further that $\Pp_y(X_{T_x}=x)=1$ for all $y<x$ where $T_x:=\inf\{t>0\,:\, X_t\ge x\}$ is the first entry time into $[x,\infty)$. This condition is, for example, satisfied if $X$ is a regular diffusion or if $X$ is a spectrally negative (i.e.\ without positive jumps) L\'evy process.
    For any measurable and locally bounded $f:\R\to [0,\infty)$ define $A_t:=\int_0^t f(X_s)\, ds$ and $M_t:=e^{-A_t}$. Assume that $A_{\infty}<\infty$ $\Pp_x$ a.s.~for all $x\in \R$, and let $u(x):=\E_x[M_{\infty}]$. Then we have for $y<x$
    \begin{equation}\label{e:regular-diffusion}
        u(y)
        =\E_y[M_{\infty}]
        =\E_y[M_{T_x}(M_{\infty}\circ \theta_{T_x})]
        \le \E_y[M_{\infty}\circ \theta_{T_x}]
        =\E_y\left[\E_{X_{T_x}} [M_{\infty}]\right]
        = u(x),
    \end{equation}
    showing that $u$ is non-decreasing.

    For $x\in \R$ we define $f_x(y):=f(y) \I_{[x,\infty)}(y)$, $A^x_t:=\int_0^t f_x(X_s)\, ds$, $M^x_t:=e^{-A^x_t}$ and $L_x:=\sup\{t>0\,:\, X_t=x\}$. Since $\lim_{t\to \infty}X_t=+\infty$, we have that $\Pp_y(L_x<\infty)=1$ for all $y$. Furthermore,
    $$
        A_{\infty}
        = \int_0^{\infty} f(X_s)\, ds
        = \int_0^{L_x}f(X_s)\, ds + \int_{L_x}^{\infty}f(X_s)\, ds
        = \int_0^{L_x}f(X_s)\, ds +  \int_{L_x}^{\infty}f_x(X_s)\, ds,
    $$
    implies that $A_{\infty}<\infty$ $\Pp_y$ a.s.~if, and only, if $A^x_{\infty}<\infty$ $\Pp_y$ a.s.
    If $u^x(y):=\E_y[M^x_{\infty}]$, then the same calculation as in \eqref{e:regular-diffusion} together with $\Pp_y(M^x_{T_x}=1)=1$ for $y<x$, shows that $u^x(y)=u^x(x)$ for all $y<x$. Since $u^x$ is non-decreasing, we conclude that $\inf_{y\in \R}u^x(y)=u^x(x)>0$. Now it follows from Proposition \ref{p:u-6}\ref{p:u-6-a} that
    $$
        \sup_{y\in \R} \E_y A^x_{\infty} \le u^x(x)^{-1} -1 <\infty.
    $$
    Since $\E_y A^x_{\infty}<\infty$ implies $\Pp_y\big(A^x_{\infty}<\infty\big)=1$, we have shown that the following four statements are equivalent:
    \begin{enumerate}
    \item [(i)] $\Pp_x\big(A_{\infty}<\infty\big)=1$ for all $x\in \R$;
    \item[(ii)] $\Pp_x\big(A^x_{\infty}<\infty\big)=1$ for all $x\in \R$;
    \item[(iii)] $\E_x A^x_{\infty} <\infty$ for all $x\in \R$;
    \item[(iv)] $\sup_{y\in \R} \E_y A^x_{\infty} <\infty$ for all $x\in \R$.
    \end{enumerate}
    This gives an alternative proof of the equivalences (i)--(iv) from \cite[Theorem 3]{KSY} (their setting is slightly more general since the state space is an arbitrary interval $(l,r)\subset \R$ and the lifetime $\zeta$ can be finite; but our proofs are easily adapted to that setting). Note that even for a bounded $f$ we cannot conclude that $\E_x A_{\infty}<\infty$, because of the lack of control of $u(x)$ as $x\to -\infty$.
\end{example}

\section{Isotropic stable L\'evy processes}\label{sec-4}
In this section we will assume that $X$ is an isotropic $\alpha$-stable L\'evy process in $\R^d$, $0<\alpha<2 \wedge d$. Recall that a L\'evy process is a stochastic process with independent and stationary increments and c\`adl\`ag paths. The transition function is uniquely determined by the characteristic function of $X_t$ which is, in the case of an isotropic stable L\'evy process, $\exp(-t|\xi|^\alpha)$.
In particular, $X_t$ has a continuous transition density $p(t,x,y) = p(t,x-y)$, $t > 0$, $x,y\in \R^d$. From the L\'evy--Khintchine representation
$$
    |\xi|^\alpha = \frac{\alpha 2^{\alpha-1}\Gamma\big(\frac{\alpha+d}2\big)}{\pi^{d/2}\Gamma\big(1-\frac\alpha 2\big)} \int_{\R^d\setminus\{0\}} (1-\cos y\cdot\xi)\,\frac{dy}{|y|^{d+\alpha}}
$$
we see that the L\'evy measure is $\tilde c(d,\alpha) |y|^{-d-\alpha}\,dy$, and because of the stationarity of the increments, we see that the L\'evy system $(N,H)$ is of the form $N(x,dy) = j(x,y)\,dy$ and $H_s(\omega)\equiv s$ with $j(x,y)=\tilde{c}(d,\alpha)|x-y|^{-d-\alpha}$. The corresponding  Dirichlet form $(\EE, \FF)$ of $X$ is
$$
    \FF=\left\{f\in L^2(\R^d,dx)\,:\, \int_{\R^d}\int_{\R^d} (f(x)-f(y))^2 j(x,y)\, dx \, dy <\infty\right\},
$$
and
$$
    \EE(f,f)=\int_{\R^d}\int_{\R^d} (f(x)-f(y))^2 j(x,y)\, dx \, dy,\quad f\in \FF.
$$

Let $F\in I_2(X)$ such that $\inf_{x,y\in \R^d} F(x,y)>-1$, and let $\widetilde{X}=(\widetilde{X}_t, \MM, \MM_t, \widetilde{\Pp}_x)$ be the corresponding  purely discontinuous Girsanov transform of $X$. It follows from \cite[Lemma 2.1 and Theorem 2.5]{S} that the semigroup of $\widetilde{X}$ is symmetric (with respect to Lebesgue measure), and that the Dirichlet form $(\widetilde{\EE}, \widetilde{\FF})$ of $\widetilde{X}$ in $L^2(\R^d, dx)$ is given by $\widetilde{\FF}=\FF$, and
$$
    \widetilde{\EE}(f,f)=\int_{\R^d}\int_{\R^d} (f(x)-f(y))^2 (1+F(x,y))j(x,y)\, dx\, dy, \quad f\in \widetilde{\FF}.
$$
Since the killing measure of $\widetilde{X}$ is zero and $1+F(x,y)$ is bounded from below and above by positive constants, we conclude that $\widetilde{X}$ is conservative. By \cite[Theorem 2.7]{S}, $\widetilde{X}$ has continuous transition densities $\tilde{p}(t,x,y)$, $t> 0$, $x,y\in \R^d$.

Since $\alpha<d$, $X$ is transient and has  the Green function $G(x,y)=c(d,\alpha)|x-y|^{\alpha-d}$, where $c(d,\alpha)=2^{-\alpha}\pi^{-d/2}\Gamma\big(\frac{d-\alpha}{2}\big)/\Gamma\big(\frac\alpha 2\big)$. In this case, $\widetilde{X}$ also admits a Green function $\widetilde{G}(x,y)$ which satisfies the following two-sided estimates, cf.~\cite[Corollary 2.8]{S}:
\begin{equation}\label{e:green-fe}
    c^{-1} G(x,y)\le \widetilde{G}(x,y) \le c\, G(x,y), \quad x,y\in \R^d,
\end{equation}
for some constant $c\ge 1$.

\begin{defn}
    A measurable function $f:\R^d \to \R$ is said to be in the \emph{Kato class} $K(X)$ if
    $$
        \lim_{t\to 0} \sup_{x\in \R} \E_x\left(\int_0^t |f(X_t)|\, dt\right) =0.
    $$
\end{defn}
If $F:\R^d\times \R^d\to\real$ is a symmetric, bounded measurable function vanishing on the diagonal, then $F\in J(X)$, cf.\ Definition \ref{d:i2x}, if, and only, if $h\in K(X)$ where
\begin{equation}\label{e:def-of-h}
    h(y):=\int_{\R^d} F(y,z) j(y,z)\, dz.
\end{equation}

\begin{defn}\label{d:iCb}
    Let $C>0$ and $\beta>0$. A bounded symmetric function $F:\R^d\times \R^d\to\R$ vanishing on the diagonal is in the class $I(C,\beta)$ if
    $$
        |F(x,y)|\le C |x-y|^{\beta} \quad\text{for all\ } x,y\in \R^d.
    $$
\end{defn}
We note that $F\in I(C,\beta)$ if, and only, if $F$ is symmetric and $|F(x,y)|\le \widehat{C}(|x-y|^{\beta}\wedge 1)$ for some constant $\widehat{C}>0$.
It is proved in \cite[Example p.~492]{S} that $I(C,\beta)\subset I_2(X)$ if $\beta >\alpha/2$. Similarly, we have the following simple result.

\begin{lemma}\label{l:kato-class}
    If $F\in I(C,\beta)$ for $\beta>\alpha$, then $F\in J(X)$.

    Consequently, if $A_t:=\sum_{s\le t}F(X_{s-},X_s)$, then $\lim_{t\to 0}\sup_{x\in \R^d} \E_x A_t=0$.
\end{lemma}
\begin{proof}
    By the properties of the L\'evy system we have
    \begin{align*}
        \E_x [A_t]
        = \E_x \left[\sum_{s\le t} F(X_{s-}, X_s) \right]
        &= \E_x \left[\int_0^t\!\!\! \int_{\R^d} F(X_{s-},y) j(X_{s-},y)\, dy \, ds \right]\\
        &= \E_x \left[\int_0^t h(X_s)\, ds\right]
    \end{align*}
    where
    \begin{align*}
        h(y)
        &=\int_{\R^d} F(y,z)j(y,z)\, dz \\
        &\le c_1 \int_{\R^d}\big(|y-z|^{\beta}\wedge 1 \big) |y-z|^{-d-\alpha}\, dz\\
        &= c_1\left(\int_{|y-z|\le 1} |y-z|^{\beta-d-\alpha}\, dz +\int_{|y-z|>1} |y-z|^{-d-\alpha}\, dz\right) = c_2 <\infty.
    \end{align*}
    Therefore, $\E_x [A_t] =\E_x\left[\int_0^t h(X_s)\, ds\right] \le c_2 t $ which implies the statement.
\end{proof}

Denote by $B(x_0,r)$ the open ball centred at $x_0$ with radius $r>0$, and by $G_{B(x_0,r)}$ the Green function of the process $X$ killed upon exiting $B(x_0,r)$. To simplify notation we write $B=B(0,1)$.
\begin{lemma}\label{l:3G-estimate}
    Let $\beta >\alpha$. There exists a constant $C_1(d,\alpha,\beta)<\infty$ such that
    $$
        \sup_{x,w\in B} \int_B\int_B \frac{G_B(x,y)G_B(z,w)}{G_B(x,w)}\, |y-z|^{\beta-\alpha-d}\, dz\, dy =C_1(d,\alpha,\beta).
    $$
\end{lemma}
\begin{proof}
    We follow the arguments from \cite[Example 2]{CS03}. From the sharp estimates of the Green function $G_B$ we get: If $\delta_B(x)=\mathrm{dist}(x, B^c)$ is the distance to the boundary, there exists a constant $c_1=c_1(d,\alpha)$ such that
    \begin{gather*}
        \frac{G_B(x,y)G_B(z,w)}{G_B(x,w)}
        \leq
        \begin{cases}
            \dfrac{c_1\,|x-w|^{d-\alpha}}{|x-y|^{d-\alpha}|z-w|^{d-\alpha}},         & |x-w|\leq\frac12 \max\{\delta_B(x), \delta_B(w)\},\\[\bigskipamount]
            \dfrac{c_2 \, |x-w|^{d-\alpha/2}}{|x-y|^{d-\alpha/2}|z-w|^{d-\alpha/2}},   & |x-w|>\frac12 \max\{\delta_B(x), \delta_B(w)\}.
        \end{cases}
    \end{gather*}
    These estimates imply that
    \begin{align*}
        &\frac{G_B(x,y)G_B(z,w)}{G_B(x,w)}\, |y-z|^{\beta-\alpha-d}\\
        &\le  c_1\left[\frac{|x-w|^{d-\alpha}}{|x-y|^{d-\alpha}|z-w|^{d-\alpha}|y-z|^{d+\alpha-\beta}} +\frac{|x-w|^{d-\alpha/2}}{|x-y|^{d-\alpha/2}|z-w|^{d-\alpha/2}|y-z|^{d+\alpha-\beta}}\right]\\
        &\le c_2\left[\frac{|x-y|^{d-\alpha}+|y-z|^{d-\alpha}+|z-w|^{d-\alpha}}{|x-y|^{d-\alpha}|z-w|^{d-\alpha}|y-z|^{d+\alpha-\beta}} +\frac{|x-y|^{d-\alpha/2}+|y-z|^{d-\alpha/2}+|z-w|^{d-\alpha/2}}{|x-y|^{d-\alpha/2}|z-w|^{d-\alpha/2}|y-z|^{d+\alpha-\beta}}\right]\\
        &= c_2 \Bigg[\frac{1}{|z-w|^{d-\alpha}|y-z|^{d+\alpha-\beta}} + \frac{1}{|x-y|^{d-\alpha}|y-z|^{d+\alpha-\beta}}\\
        &\quad\qquad\mbox{} +\frac{1}{|z-w|^{d-\alpha/2}|y-z|^{d+\alpha-\beta}} + \frac{1}{|x-y|^{d-\alpha/2}|y-z|^{d+\alpha-\beta}}\\
        &\quad\qquad\mbox{} + \frac{1}{|x-y|^{d-\alpha}|z-w|^{d-\alpha}|y-z|^{2\alpha-\beta}} +\frac{1}{|x-y|^{d-\alpha/2}|z-w|^{d-\alpha/2}|y-z|^{3\alpha/2-\beta}}\Bigg]
    \end{align*}
    where $c_2=c_2(d,\alpha)$. The integrals of the first four terms are estimated in the same way, so we only do the first one. Since $\beta>\alpha>0$, we have that
    \begin{align*}
        \int_B\int_B \frac{1}{|z-w|^{d-\alpha}|y-z|^{d+\alpha-\beta}}\, dz\, dy
        &\le  \int_{B(w,2)}\frac{1}{|w-z|^{d-\alpha}}\left( \int_{B(z,2)}\frac{1}{|y-z|^{d+\alpha-\beta}}\, dy\right)\, dz \\ &=c_3(d,\alpha,\beta)<\infty.
    \end{align*}
    In order to estimate the fifth term we use H\"older's inequality with $\frac{2\alpha-\beta}{d}<\frac{1}{q} < \frac{\alpha}{d}$ (note that $2\alpha-\beta<\alpha$) and $\frac{1}{p}=1-\frac{1}{q}>1-\frac{\alpha}{d}$ to get
    \begin{align*}
        \int_B \frac{dz}{|z-w|^{d-\alpha}|y-z|^{2\alpha-\beta}}
        &\le  \left(\int_B |z-w|^{(\alpha-d)p}\, dz\right)^{1/p}\, \left(\int_B |y-z|^{(\beta-2\alpha)q}\, dz\right)^{1/q}\\
        &\le  \left(\int_{B(w,2)} |z-w|^{(\alpha-d)p}\, dz\right)^{1/p}\, \left(\int_{B(y,2)} |y-z|^{(\beta-2\alpha)q}\, dz\right)^{1/q}\\
        &= c_4(d,\alpha,\beta)<\infty.
    \end{align*}
    Therefore
    $$
        \int_B\int_B \frac{dz\,dy}{|x-y|^{d-\alpha}|z-w|^{d-\alpha}|y-z|^{2\alpha-\beta}}
        \le c_4(d,\alpha,\beta)\int_{B(x,2)}\frac{dy}{|x-y|^{d-\alpha}}
        \le c_5(d,\alpha,\beta).
    $$
    The integral of the sixth term is estimated in the same way.
\end{proof}

\begin{lemma}\label{l:key-lemma}
    For every $\epsilon >0$ there exists some $r_0=r_0(C,d,\alpha,\beta,\epsilon)>0$ such that for every $r\le r_0$, $x,w\in B(x_0,r)$ and $F\in I(C,\beta)$ it holds that
    \begin{equation}\label{e:key-lemma-1}
        \int_{B(x_0,r)}\int_{B(x_0,r)}\frac{G_{B(x_0,r)}(x,y)G_{B(x_0,r)}(z,w)}{G_{B(x_0,r)}(x,w)}\, |F(y,z)| \, |y-z|^{-\alpha-d}\, dz\, dy <\epsilon,
    \end{equation}
\end{lemma}
\begin{proof}
    By the scaling property and the spatial homogeneity of a stable L\'evy process we see
    $$
        G_{B(x_0,r)}(x,y)=r^{\alpha-d} G_B(r^{-1}(x-x_0), r^{-1}(y-x_0)),
    $$
    where $B=B(0,1)$. Therefore,
    \begin{align*}
        &\int\limits_{B(x_0,r)}\int\limits_{B(x_0,r)}\frac{G_{B(x_0,r)}(x,y)G_{B(x_0,r)}(z,w)}{G_{B(x_0,r)}(x,w)}\, F(y,z)\,\frac{dz\,dy}{|y-z|^{\alpha+d}}\\
        &\le C r^{\alpha-d} \int\limits_{B(x_0,r)}\int\limits_{B(x_0,r)} \frac{G_B\big(\tfrac 1r(x-x_0), \tfrac 1r(y-x_0)\big) G_B\big( \tfrac 1r(z-x_0), \tfrac 1r(w-x_0)\big)}{G_B\big(\tfrac 1r(x-x_0), \tfrac 1r(w-x_0)\big)} \, \frac{|y-z|^{\beta}\,dz\,dy}{|y-z|^{d+\alpha}}.
    \end{align*}
    Using the change of variables $y'=r^{-1}(y-x_0)$, $z'=r^{-1}(z-x_0)$, where $x'=r^{-1}(x-x_0)$, $w'=r^{-1}(w-x_0)$ and by Lemma \ref{l:3G-estimate}, the last expression is equal to
    \begin{align*}
        &Cr^{\alpha-d} \int_B\int_B \frac{G_B(x',y')G_B(z',w')}{G_B(x',w')} r^{\beta}|y'-z'|^{\beta} r^{-d-\alpha} |y'-z'|^{-d-\alpha} r^{2d}\, dz'\, dy'\\
        &= Cr^\beta \int_B\int_B \frac{G_B(x',y')G_B(z',w')}{G_B(x',w')} |y'-z'|^{\beta-\alpha-d}\, dz'\, dy'\\
        &= C C_1(d,\alpha,\beta) r^\beta.
    \end{align*}
    Now we choose $r_0=r_0(C,d,\alpha, \beta, \epsilon) >0$ such that $C C_1(d,\alpha,\beta) r_0^\beta<\epsilon$.
\end{proof}

Let $F:\R^d\times \R^d\to\R$ be bounded and symmetric, set $A_t:=\sum_{s\le t}F(X_{s-}, X_s)$ and denote by $\tau_{B(x_0,r)}=\inf\{t>0:\, X_t\notin B(x_0,r)\}$ the first exit time of $X$ from the ball $B(x_0,r)$. For $x,w\in B(x_0,r)$, let $\Pp_x^w$ denote the law of the $h$-transformed killed process $X^{B(x_0,r)}$ with respect to the excessive function $G_{B(x_0,r)}(\cdot, w)$---this is the process $X^{B(x_0,r)}$ conditioned to die at $\{w\}$. By \cite[Proposition 3.3]{CS03} it holds that
\begin{align*}
    \E_x^w &\left[\sum_{s\le t}F(X_{s-}^{B(x_0,r)}, X_s^{B(x_0,r)})\right]\\
    &=\E_x\left[\int_0^t \int_{B(x_0,r)} \frac{F(X_s^{B(x_0,r)},z)G_{B(x_0,r)}(z,w)}{G_{B(x_0,r)}(x,w)}\, |X_s^{B(x_0,r)}-z|^{-d-\alpha}\, dz\, ds\right].
\end{align*}
This formula remains valid if we replace $t$ with $\tau_{B(x_0,r)}$. On $[0,\tau_{B(x_0,r)})$ the killed process coincides with $X$, and so
\begin{align}
    \E_x^w &\left[\sum_{s\le \tau_{B(x_0,r)}}F(X_{s-}, X_s)\right] \nonumber \\
    &= \E_x\left[\int_0^{\tau_{B(x_0,r)}} \int_{B(x_0,r)} \frac{F(X_s,z)G_{B(x_0,r)}(z,w)}{G_{B(x_0,r)}(x,w)}\, |X_s-z|^{-d-\alpha}\, dz\, ds\right] \nonumber\\
    &=  \int_{B(x_0,r)} \int_{B(x_0,r)} \frac{G_{B(x_0,r)}(x,y)G_{B(x_0,r)}(z,w)}{G_{B(x_0,r)}(x,w)}\, F(y,z) |y-z|^{-\alpha-d}\, dz\, dy.\label{e:levy-system-conditioned}
\end{align}

\begin{lemma}\label{l:estimate-for-killing}
    Assume that $F$ is a non-negative function in $I(C,\beta)$, let $\epsilon >0$ and denote by $r_0=r_0(C,d,\alpha,\beta,\epsilon)>0$ the constant from Lemma \ref{l:key-lemma}. Then for $r\le r_0$ it holds that
    $$
        e^{-\epsilon}\le \E_x^w \left[e^{-A_{\tau_{B(x_0,r)}}}\right] \le 1.
    $$
\end{lemma}
\begin{proof}
    Let $r\le r_0$ and set $\tau=\tau_{B(x_0,r)}$. By \eqref{e:levy-system-conditioned} and Lemma \ref{l:key-lemma} we see that $\E_x^w[A_{\tau}]<\epsilon$. By Jensen's inequality, it follows that
    \begin{gather*}
        e^{-\epsilon}\le e^{-\E_x^w[A_{\tau}]}\le \E_x^w[e^{-A_{\tau}}] \leq 1.
    \qedhere
    \end{gather*}
\end{proof}

\begin{remark}
    If we do not assume in Lemma \ref{l:estimate-for-killing} that $F$ is non-negative, we could use Khas'minskii's lemma, see~e.g.~\cite[Lemma 3.7]{chung-zhao}, to get $\E_x^w \left[e^{A_{\tau_{B(x_0,r)}}}\right]\le (1-\epsilon)^{-1}$.
\end{remark}

Let $F$ be a non-negative, symmetric function on $\R^d \times\R^d$ and $D\subset \R^d$ a bounded open set. We say that a non-negative function $u:\R^d\to [0,\infty)$ is \emph{$F$-harmonic in $D$} if for every open set $V\subset \overline{V}\subset D$ the following mean-value property holds:
$$
    u(x)=\E_x\left[e^{-A_{\tau_V}} u(X_{\tau_V})\right],\quad \text{for all }x\in V.
$$
The function $u$ is \emph{regular $F$-harmonic in $D$}  if $u(x)=\E_x\left[e^{-A_{\tau_D}} u(X_{\tau_D})\right]$ for all $x\in D$. A standard argument using the strong Markov property shows that any regular $F$-harmonic function $u$ is also $F$-harmonic in $D$.

We will now prove the Harnack inequality for non-negative $F$-harmonic functions. For a continuous additive functional $A_t=\int_0^t f(X_s)\, ds$ with $f\in K(X)$ an analogous result has been proved in \cite[Theorem 4.1]{BB00}. Our argument is a modification of that proof. Recall that the Poisson kernel of the ball $B(0,r)$ is given by
$$
    P_r(x,z)
    = \widehat{c}(\alpha,d)  \left(\frac{r^2-|x|^2}{|z|^2-r^2}\right)^{\alpha/2}\, |x-z|^{-d},\quad |x|<r, |z|>r,
$$
where $\widehat{c}(\alpha,d)=\pi^{1+d/2}\Gamma(d/2) \sin(\pi \alpha/2)$.

\begin{thm}\label{t:hi}
    Let $D\subset \R^d$ be a bounded open set and $K\subset D$ a compact subset of $D$. There exists a constant $C_2=C_2(C,\beta, d,\alpha, D, K)>0$ such that for every $F\in I(C,\beta)$ and every $u:\R^d\to [0,\infty)$ which is $F$-harmonic in $D$, it holds that
    $$
        C_2^{-1}u(x) \le u(y) \le C_2 u(x),\quad x,y \in K.
    $$
\end{thm}
\begin{proof}
    Set $\delta_K=\mathrm{dist}(K, D^c)$ and $\rho_0=r_0\wedge \delta_K/2$ where $r_0=r_0(C,d,\alpha,\beta, 1/2)$ is the constant from Lemma \ref{l:key-lemma} with $\epsilon=1/2$. Let $x\in K$, $0<r\le \rho_0$ and $B=B(x,r)$. By \cite[(2.15)]{BB00},
    see also \cite[Theorem 2.4]{CS03b},
    $$
        u(y)
        =\E_y\left[u(X_{\tau_B})e^{-A_{\tau_B}}\right]
        =\E_y\left[u(X_{\tau_B})\E_y^{X_{\tau_B-}} \left[e^{-A_{\zeta}}\right]\right],
    $$
    where $\zeta=\tau_{B\setminus \{v\}}$ and $v=X_{\tau_B-}$. By Lemma \ref{l:estimate-for-killing}, $\frac12\le \E_x^{X_{\tau_B-}} \left[e^{-A_{\zeta}}\right] \le 1$, implying that
    \begin{equation}\label{e:hi-1}
        \frac12 \E_y[u(X_{\tau_B})]\le u(y)\le \E_y[u(X_{\tau_B})].
    \end{equation}
    If $|y-x|< \frac 12r$, then
    \begin{align*}
        \E_x [u(X_{\tau_B})] &= \int_{B^c}P_r(0,z-x) u(z)\, dz \\
        &\le  \sup_{|z-x|>r} \frac{P_r(0,z-x)}{P_r(y-x,z-x)}\, \int_{B^c}P_r(y-x,z-x) u(z)\, dz\\
        &\le 3^{d+1} \E_y [u(X_{\tau_B})].
    \end{align*}

    Here $P_r(\cdot, \cdot)$ denotes the Poisson kernel of the ball, and the last estimate follows from the explicit formula for $P_r$. Similarly,
    $$
        \E_x [u(X_{\tau_B})]\ge 3^{-d-1}\E_y [u(X_{\tau_B})].
    $$
    Combining the last two estimates with \eqref{e:hi-1} yields
    \begin{equation}\label{e:hi-2}
        3^{-d-3} u(y) \le u(x)\le 3^{d+3} u(y),\quad y\in B(x,r/2).
    \end{equation}
    In particular, \eqref{e:hi-2} holds for $r=\rho_0$.

    Now pick $z\in K$ such that $|z-x|\ge \frac 12\rho_0$ and set $\hat B:=B(z,\rho_0/4)$. With $B=B(x,\rho_0/4)$ we have that $B\cap \hat B=\emptyset$. By using \eqref{e:hi-1} in the first line, \eqref{e:hi-2} in the fourth line (with $r=\rho_0/2$), and the estimate of the Poisson kernel in the fifth line of the calculation below, we get
    \begin{align*}
        u(z)
        &\ge \frac12 \E_z[u(X_{\tau_{\hat B}})]\\
        &=   \frac12 \int_{\hat B^c} P_{\rho_0/4}(0,y-z)u(y)\, dy\\
        &\ge \frac12 \int_B P_{\rho_0/4}(0,y-z)u(y)\, dy \\
        &\ge  \frac12 3^{-d-3}u(x)\int_B P_{\rho_0/4}(0,y-z)\, dy \\
        &\ge  \frac12 3^{-d-3}u(x) \widehat{c}(\alpha,d)(\rho_0/4)^{\alpha} \int_{B}|y-z|^{-d-\alpha}\, dy\\
        &\ge  3^{-d-3}2^{-1-2\alpha}\widehat{c}(\alpha,d)\rho_0^{\alpha}\big|B\big(0,\rho_0/4\big)\big| \left(\frac{3|z-x|}{2}\right)^{-d-\alpha} u(x)\\
        &\ge  3^{-2d-3-\alpha}2^{-d-\alpha-1}\widehat{c}(\alpha,d)\big|B(0,1)\big| \rho_0^{d+\alpha} (\diam K)^{-d-\alpha} u(x)\\
        &= C(\alpha,d) \left(\frac{\rho_0}{\diam K}\right)^{d+\alpha}u(x) =c u(x).
    \end{align*}
    Similarly, $u(x)\ge c u(z)$. Together with \eqref{e:hi-2} this proves the theorem.
\end{proof}

\begin{remark}
    The constant $C_2$ depends on $d$, $\alpha$ and the ratio $\rho_0/\diam K$.
\end{remark}

\begin{lemma}\label{l:harmonic-semilocal}
    Let $D\subset \R^d$ be a bounded open set and assume that $F, \Phi$
    are two non-negative symmetric functions on $\R^d\times\R^d$ vanishing on the diagonal such that
    $$
        F(x,y)=\Phi(x,y)
        \quad\text{for all}\quad
        (x,y)\in (D\times\R^d)\cup (\R^d\times D).
    $$
    Then $u$ is (regular) $F$-harmonic in $D$ if, and only, if $u$ is (regular) $\Phi$-harmonic in $D$.
\end{lemma}
\begin{proof}
    Let $A_t=\sum_{s\le t}F(X_{s-},X_s)$ and $B_t=\sum_{s\le t}\Phi(X_{s-},X_s)$. For any open $V\subseteq D$ we have  $X_{s-}\in V\subset D$, and so $F(X_{s-},X_s)= \Phi(X_{s-},X_s)$ for all $s\leq\tau_V$. Thus, $A_{\tau_V}
        =  \sum_{s\le \tau_V} F(X_{s-},X_s)
        =  \sum_{s\le \tau_V} \Phi (X_{s-},X_s)
        =B_{\tau_V}$.
    Therefore
    $$
        \E_x\left[e^{-A_{\tau_V}}u(X_{\tau_V})\right]=\E_x\left[e^{-B_{\tau_V}}u(X_{\tau_V})\right],
    $$
    proving the claim.
\end{proof}

    For $R>0$ set $u_R(x):=u(Rx)$, $F_R(x,y):=F(Rx,Ry)$, $D_R:= R\cdot D := \{Rx:\, x\in D\}$, and define the additive functional $A^R_t:=\sum_{s\le t}F_R(X_{s-}, X_s)$.
\begin{lemma}\label{l:harmonic-scaling-b}
    Let $R>0$ and assume that $u$ is regular $F$-harmonic in $D_R$, i.e.
    \begin{equation}\label{e:harmonic-scaling-1}
        u(x)=\E_x\left[e^{-A_{\tau_{D_R}}}u(X_{\tau_{D_R}})\right]\quad \text{for all\ \ } x\in D_R.
    \end{equation}
    Then $u_R$ is regular $F_R$-harmonic in $D$, i.e.
    \begin{equation}\label{e:harmonic-scaling-2}
        u_R(x)=\E_x\left[e^{-A^R_{\tau_D}}u_R(X_{\tau_D})\right]\quad \text{for all\ \ } x\in D.
    \end{equation}
\end{lemma}
\begin{proof}
    We begin with some scaling identities. Note that the $\Pp_x$-distribution of $(R X_t)_{t\ge 0}$ is equal to the $\Pp_{Rx}$-distribution of $(X_{R^{\alpha}t})_{t\ge 0}$. From this identity it follows that the $\Pp_x$-distribution of the pair $(\tau_D, R X_{\tau_D})$ is equal to the $\Pp_{Rx}$-distribution of  $(R^{-\alpha}\tau_{D_R},X_{\tau_{RD}})=(R^{-\alpha}\tau_{D_R},X_{\tau_{D_R}})$. Using these scaling identities in the second line, a change of variables in the third line and \eqref{e:harmonic-scaling-1} in the fourth line below, we get
    \begin{align*}
        \E_x\left[e^{-A^R_{\tau_D}}u_R(X_{\tau_D})\right]
        &=  \E_x\left[e^{-\sum_{s\leq \tau_D}F(R X_{s-}, R X_s)}u_R(X_{\tau_D})\right]\\
        &=  \E_{Rx}\left[e^{-\sum_{s\leq R^{-\alpha}\tau_{D_R}} F(X_{R^\alpha s-},X_{R^\alpha s})} u(X_{\tau_{D_R}})\right]\\
        &=  \E_{Rx} \left[e^{-\sum_{s\leq \tau_{D_R}}F(X_{s-},X_s)} u(X_{\tau_{D_R}})\right]\\
        &=  u(Rx)
        =  u_R(x).
    \qedhere
    \end{align*}
    \end{proof}

\begin{lemma}\label{l:fuchsian}
Assume that $F:\R^d\times\R^d \to [0,\infty)$ is symmetric, bounded and satisfies
\begin{equation}\label{e:fuchsian}
    F(x,y)\le C \frac{|x-y|^{\beta}}{1+|x|^{\beta}+|y|^{\beta}}
    \qquad\text{for all\ \ } x,y\in\R^d.
\end{equation}
\begin{enumerate}
\item\label{l:fuchsian-a}
    $F_R$ is symmetric, bounded and satisfies $F_R(x,y)\le C|x-y|^{\beta}$ for all $(x,y)\in (B(0,1)^c\times \R^d)\cup (\R^d\times B(0,1)^c)$.
\item\label{l:fuchsian-b}
    For a bounded open set $D\subset B(0,1)^c$ let
    $$
        \widehat{F}_R(x,y)
        =
        \begin{cases}
        F_R(x,y)  & \text{if\ \ } (x,y)\in (D\times \R^d)\cup (\R^d\times D)\\
        0         & \text{otherwise.}
        \end{cases}
    $$
        Then $\widehat{F}_R$ is symmetric, bounded and satisfies $\widehat{F}_R(x,y)\le C|x-y|^{\beta}$ for all $x,y\in\R^d$.
\end{enumerate}
\end{lemma}
\begin{proof}
\ref{l:fuchsian-b} follows directly from \ref{l:fuchsian-a}.
\ref{l:fuchsian-a} Symmetry and boundedness are clear. For $|x|\ge 1$ or $|y|\ge 1$ we have
\begin{gather*}
    F_R(x,y)
    = F(Rx, R y)
    \le C\frac{|Rx-Ry|^{\beta}}{1+|Rx|^{\beta}+|Ry|^{\beta}}
    =C \frac{|x-y|^{\beta}}{R^{-\beta}+|x|^{\beta}+|y|^{\beta}}
    \le C|x-y|^{\beta}.
\qedhere
\end{gather*}
\end{proof}

\begin{remark}
\begin{enumerate}
\item
    Note that \eqref{e:fuchsian} implies the condition
    \begin{equation}\label{e:fuchsian-2}
        F(x,y)\le C \frac{|x-y|^{\beta}}{1+|x|^{\beta}} \qquad \text{for all }x,y\in\R^d.
    \end{equation}
    Conversely, if $F$ is symmetric and \eqref{e:fuchsian-2} holds, then $F$ satisfies \eqref{e:fuchsian} with $2C$ instead of $C$. Indeed, by symmetry, \eqref{e:fuchsian-2} is valid with $y$ instead of $x$ in the denominator, and thus
    $$
        F(x,y)
        \le C |x-y|^{\beta}\min\left\{\frac{1}{1+|x|^{\beta}},\: \frac{1}{1+|y|^{\beta}}\right\}
        \le C |x-y|^{\beta}\frac{2}{1+|x|^{\beta}+|y|^{\beta}}.
    $$

\item
    Note that the statement of Lemma \ref{l:fuchsian} \ref{l:fuchsian-b} can be rephrased as $\widehat{F}_R\in I(C,\beta)$ for all $R>0$.
    This will be crucial in the proof of Theorem~\ref{t:finite-expectation}.
\end{enumerate}
\end{remark}

For a Borel set $C\subset \R^d$ let $T_C=\inf\{t>0:\, X_t\in C\}$ be its hitting time. If $0<a<b$, let $V(0,a,b):=\{x\in \R^d:\, a<|x|<b\}$ be the open annulus, and denote by $\overline{V}(0,a,b)$ its closure.
\begin{lemma}\label{l:infinite-hitting}
    Let $(R_n)_{n\ge 1}$ be a strictly increasing sequence of positive numbers such that $\lim_{n\to \infty}R_n=\infty$, and let $V_n:= \overline{V}(0,R_n,2R_n)=\{x\in \R^d\, :\, R_n \le |x|\le 2R_n\}$. Then
    $$
        \Pp_x\left(\limsup_{n\to \infty}\, \{T_{V_n}<\infty\}\right) =1
        \quad\text{for all\ \ } x\in\R^d.
    $$
\end{lemma}
Lemma \ref{l:infinite-hitting} says that $\Pp_x\left(\{T_{V_n}<\infty\}\ \text{infinitely often}\right)=1$, i.e.~with $\Pp_x$ probability $1$, the process $X$ visits infinitely many of the sets $V_n$.
\begin{proof}[Proof of Lemma \ref{l:infinite-hitting}]
    Let $C_k:=\bigcup_{n\geq k} V_n$. By \cite[Proposition 2.5]{HN13}, $\Pp_x(T_{C_k}<\infty)=1$ for every $x\in \R^d$ and $k\geq 1$. Obviously,
    $$
        \big\{T_{C_k} < \infty\big\}
        = \big\{T_{\bigcup_{n\geq k} V_n} < \infty\big\}
        = \bigcup_{n\geq k}\big\{T_{V_n} < \infty\big\}.
    $$
    Since this inclusion holds for all $k\geq 1$, we get
    $$
        \bigcap_{k\geq 1}\big\{T_{C_k} < \infty\big\}
        \subset \bigcap_{k\geq 1}\bigcup_{n\geq k}\big\{T_{V_n} < \infty\big\}
        = \limsup_{n\to\infty}\big\{T_{V_n} < \infty\big\}.
    $$
    Since $\Pp_x(T_{C_k} < \infty)=1$ we see
    \begin{gather*}
        1 =  \Pp_x\bigg(\bigcap_{k\geq 1}\big\{T_{C_k} < \infty\big\}\bigg)
        \leq \Pp_x\left(\limsup_{n\to\infty}\big\{T_{V_n} < \infty\big\}\right)
        = \Pp_x\big(\{T_{V_n}<\infty\}\ \text{i.o.}\big).
    \qedhere
    \end{gather*}
\end{proof}

\begin{thm}\label{t:finite-expectation}
    Assume that $F:\R^d\times\R^d \to [0,\infty)$ is bounded, symmetric and satisfies condition \eqref{e:fuchsian} with  $\beta > \alpha $. Let $A_t:=\sum_{s\le t}F(X_{s-}, X_s)$. If $\Pp_x(A_{\infty}<\infty)=1$ for all $x\in \R^d$, then $\sup_{x\in \R^d} \E_x [A_{\infty}]<\infty$.
\end{thm}
\begin{proof}
    First note that $F(x,y)\le C|x-y|^{\beta}$ for all $x,y\in \R^d$, hence $F\in I(C,\beta)$ and, by Lemma \ref{l:kato-class}, $F\in J(X)$. Let $M_t=e^{-A_t}$ and $u(x):=\E_x [M_{\infty}]$. Since $X$ is a strong Feller process, it follows from  Proposition \ref{p:u-7}\ref{p:u-7-b} and Lemma \ref{l:kato-class} that $u$ is continuous. Moreover, by Lemma \ref{l:u-4}, $u$ is regular $F$-harmonic in every bounded open set $G\subset \R^d$. Finally, by Lemma \ref{l:u-2}, we have that $\lim_{t\to \infty}u(X_t)=1$ $\Pp_x$ a.s. According to Proposition \ref{p:u-6}\ref{p:u-6-b}, in order to prove that $\sup_{x\in \R^d}\E_x [A_{\infty}]<\infty$ it suffices to show that $u$ is bounded from below by a strictly positive constant.

    Let $D=V(0,1,5)=\{x\in \R^d:\, 1<|x|<5\}$. For $R>1$ set $D_R:=R\cdot D$, $u_R(x)=u(Rx)$, $F_R(x,y)=F(Rx,Ry)$ and $\widehat{F}_R(x,y)=F_R(x,y)$ for $(x,y)\in (D\times \R^d)\cup (\R^d\times D)$, and $0$ otherwise. Since $u$ is regular $F$-harmonic
    in $D_R$, we get from Lemma \ref{l:harmonic-scaling-b} that $u_R$ is regular $F_R$-harmonic in $D$. Since $F_R(x,y)=\widehat{F}_R(x,y)$ for all $(x,y)\in (D\times \R^d)\cup (\R^d \times D)$, $u_R$ is by Lemma \ref{l:harmonic-semilocal} also regular $\widehat{F}_R$-harmonic in $D$. Moreover, by Lemma \ref{l:fuchsian}, $\widehat{F}_R\in I(C,\beta)$. Hence it follows from Theorem \ref{t:hi} that there exists a constant $c>1$ depending only on $d$ and $\alpha$ such that
    $$
        c^{-1}u_R(y)\le u_R(x) \le c u_R(y)\quad \text{for all\ \ }x,y\in \overline{V}(0,2,4).
    $$
    This can be written as
    \begin{equation}\label{e:hi-scaled}
        c^{-1}u(y)\le u(x)\le c u(y) \quad \text{for all\ \ }x,y\in \overline{V}(0,2R,4R).
    \end{equation}
    Since $u$ is continuous, lower boundedness follows if we can show that $\liminf_{|x|\to\infty}u(x)>0$. Suppose not; then there exists a sequence $(x_n)_{n\ge 1}$ in $\R^d$ such that $|x_n|\to \infty$ and $\lim_{n\to \infty} u(x_n)=0$. Let $V_k=\overline{V}(0,2^k,2^{k+1})$. Without loss of generality we may assume that there exists an increasing sequence $(k_n)_{n\geq 1}$ such that $x_n\in V_{k_n}$ for every $n\ge 1$. By Lemma \ref{l:infinite-hitting}, with $\Pp_x$-probability 1, $X$ hits infinitely many sets $V_{k_n}$. Hence, for $\Pp_x$ a.e.~$\omega$ there exists a subsequence $(n_l(\omega))$  and a sequence of times $t_l(\omega)$ such that $X_{t_l(\omega)}(\omega)\in V_{k_{n_l(\omega)}}$. By \ref{e:hi-scaled} we get that
    $$
        c^{-1}u(X_{t_l(\omega)}(\omega))\le u(x_{n_l(\omega)}) \le c u(X_{t_l(\omega)}(\omega)).
    $$
    Since, by assumption, $\lim_{l\to \infty}u(x_{n_l(\omega)})=0$ we get that $\lim_{l\to \infty}u(X_{t_l(\omega)}(\omega))=0$. But this is a contradiction with $\lim_{t\to \infty}u(X_t)=1$ $\Pp_x$ a.s.

    We conclude that $\liminf_{n\to \infty} u(x_n)\ge c^{-1}$. This finishes the proof.
\end{proof}

\begin{remark}\label{r:infinite-expectation}
    For $\alpha <\beta$ let
    $$
        F(x,y):=\frac{|x-y|^{\beta}}{1+|x|^{\beta}+|y|^{\beta}},\quad x,y\in \R^d.
    $$
    Then $F$ is non-negative, bounded, symmetric, satisfies \eqref{e:fuchsian} and $F\in  J(X)\subset I_2(X) $, cf.\ the remark following Definition \ref{d:iCb}. Let $A_t:=\sum_{s\le t}F(X_{s-},X_s)$. Then $\E_x [A_t]<\infty$ for all $t\ge 0$. On the other hand, it is not difficult to show that $\E_x [A_{\infty}]=\infty$. Hence, the statement of Theorem \ref{t:finite-expectation} is not void---there do exist functions satisfying all conditions of the theorem (except $A_{\infty}<\infty$ a.s.) but still $\E_x [A_{\infty}] =\infty$.
\end{remark}

\begin{remark}\label{r:finite-expectation}
    Suppose that $f:\R^d\to [0,\infty)$ is a measurable function satisfying the condition
    \begin{equation}\label{e:fuchsian-cont}
        f(x)\le \frac{C}{1+|x|^{\alpha}}\quad \text{for all\ \ }x\in \R^d.
    \end{equation}
    Let $A_t:=\int_0^t f(X_s)\, ds$. Using similar arguments as in the proof of Theorem \ref{t:finite-expectation} we can show that $\Pp_x(A_{\infty}<\infty)=1$ implies that $\E_x [A_{\infty}]<\infty$. Indeed, the analogues of Lemmas \ref{l:key-lemma} and \ref{l:estimate-for-killing} and Theorem \ref{t:hi} are given in \cite[Lemma 3.4, Lemma 3.5, Theorem 4.1]{BB00}, the scaling Lemma \ref{l:harmonic-scaling-b} is proved in the same way, the counterpart of Lemma \ref{l:fuchsian} shows that $\widehat{f}_R(x):=R^{\alpha}f(Rx)$  is bounded by $C$ for all $|x|\ge 1$. The rest of the argument is exactly the same as in the proof of Theorem \ref{t:finite-expectation}.
\end{remark}

    In the next result we show that the condition \eqref{e:fuchsian} is essential for validity of the Theorem \ref{t:finite-expectation}.
\begin{thm}\label{t:infinite-expectation}
    There exists a non-negative bounded function $F\in I(1,\beta)$ such that $A_{\infty}:=\sum_{s>0} F(X_{s-}, X_s)<\infty$ $\Pp_x$ a.s., but $\E_x [A_{\infty}]=\infty$.
\end{thm}
\begin{proof}
    Fix $\gamma$ and $\beta$ so that $0<\gamma <\alpha <\beta$ and $\alpha- \gamma< \frac 12$.  Let $(x_n)_{n\ge 1}$ be a sequence of
    points in $\R^d$ such that $|x_n|=2^{nd/(\alpha-\gamma)}$ and let $r_n=2^{-n}|x_n|+1$. Note that $|x_n|\ge 2^{2nd}>2^n$. Consider the family of balls $\{B(x_n,r_n)\}_{n\ge 1}$. By \cite[Lemma 2.5]{MV},
    $$
        \Pp_0(T_{B(x_n,r_n)}<\infty)
        \le \left(\frac{r_n}{|x_n|}\right)^{d-\alpha}
        =\left(\frac{2^{-n}|x_n|+1}{|x_n|}\right)^{d-\alpha}
        \le (2^{-n}+2^{-n})^{d-\alpha}=2^{(1-n)(d-\alpha)}.
    $$
    Hence,
    $\sum_{n\ge 1}\Pp_0(T_{B(x_n,r_n)}<\infty) <\infty$, implying by the Borel--Cantelli lemma that $\Pp_0(\{T_{B(x_n,r_n)}<\infty\} \text{\ i.o.})=0$. Therefore, $X$ hits $\Pp_0$ a.s.\ only finitely many balls $B(x_n,r_n)$. Let $C:=\bigcup_{n\ge 1}B(x_n,r_n)$.

    Define a symmetric bounded function $F:\R^d\times \R^d \to [0,\infty)$ by
    $$
        F(y,z)
        :=\begin{cases}
            \dfrac{|y-z|^{\beta}}{|y|^{\gamma}+|z|^{\gamma}},& y,z\in B(x_n,r_n) \text{\ for some\ } n, \;  |y-z|\le 1\\[\bigskipamount]
            0, & \text{otherwise}.
        \end{cases}
    $$
    Note that $F(y,z)\le |y-z|^{\beta}\wedge 1$ for all $y,z\in \R^d$. Thus, $F\in J(X)$ and $F\in I(1,\beta)$.

    Let $A_t:=\sum_{s\le t}F(X_{s-},X_s)$, $t\ge 0$. Then $\E_0 [A_t]<\infty$ implying that $\Pp_0(A_t<\infty)=1$ for all $t>0$.  Since $X$ visits only finitely many balls $B(x_n, r_n)$, the last exit time from the union $\bigcup_{n\ge 1}B(x_n,r_n)$ is finite, hence $\Pp_0(A_{\infty}<\infty)=1$. Since $\{A_{\infty}<\infty\}\in \II$, the argument at the end of Section \ref{sec-2} shows that $\Pp_x(A_{\infty}<\infty)=1$ for all $x\in \R^d$.

    Further,
    \begin{align*}
    \E_x [A_{\infty}]
    &=\E_x \left[\sum_{s>0}F(X_{s-},X_s)\right]\\
    &=\E_x \left[\int_0^{\infty}\!\!\! \int_{\R^d} F(X_{s-},z)j(X_{s-},z)\, dz \, ds\right]\\
    &=\E_x \left[\int_0^{\infty} h(X_s)\, ds\right]
    =Gh(x)
    =c(\alpha,d) \int_{\R^d}h(y)|x-y|^{\alpha-d}\,dy,
    \end{align*}
    where
    $$
        h(y)
        :=\int_{\R^d}F(y,z)j(y,z)\, dz
        =\tilde{c}(\alpha,d)\int_{\R^d}F(y,z) |y-z|^{-d-\alpha}\, dz.
    $$
    If $y\notin C$, then $F(y,\cdot)=0$, implying that $h(y)=0$. Let $y\in B(x_n,r_n-1)$. Then $|y|\le 2|x_n|$ and if $z$ satisfies $|z-y|<1$, then $z\in B(x_n,r_n)$ and also $|z|\le 2|x_n|$. Therefore,
\begin{align*}
    h(y)
    &=  \tilde{c}(\alpha,d)\int_{z\in B(x_n,r_n), |z-y|\le 1} \frac{|y-z|^{\beta}}{|y|^{\gamma}+|z|^{\gamma}}\, |y-z|^{-d-\alpha}\, dz\\
    &\ge  c_1 \int_{|z-y|\le 1} \frac{|y-z|^{\beta-d-\alpha}}{|x_n|^{\gamma}}\, dz \ge c_2 |x_n|^{-\gamma}.
\end{align*}
    In the last inequality we have used that $0 < \int_{|z-y|\le 1}|y-z|^{\beta-d-\alpha}\, dz <\infty$. Hence, for $|x|\le 1$ we have $|x-y|\leq 4|x_n|$, so
\begin{align*}
    Gh(x)
    &= c(\alpha,d) \sum_{n\ge 1}\int_{B(x_n,r_n)}  h(y)|x-y|^{\alpha-d}\, dy\\
    &\ge   c(\alpha,d) \sum_{n\ge 1}\int_{B(x_n,r_n-1)}  h(y)|x-y|^{\alpha-d}\, dy\\
    &\ge  c_3 \sum_{n\ge 1} |x_n|^{-\gamma} \int_{B(x_n,r_n-1)}|x_n|^{\alpha-d}\, dy\\
    &= c_4 \sum_{n\ge 1} |x_n|^{-\gamma+\alpha-d} (r_n-1)^d\\
    &= c_4 \sum_{n\ge 1} |x_n|^{-\gamma+\alpha-d} (2^{-n}|x_n|)^d\\
    &= c_4 \sum_{n\ge 1} 2^{nd}2^{-nd} = \infty.
\end{align*}
This implies that $Gh\equiv \infty$.
\end{proof}

\section{Proofs of Theorems~\ref{t:2} and \ref{t:3}}
Throughout this section we assume that $X$ is an isotropic $\alpha$-stable L\'evy process such that $\alpha < 2\wedge d$.  We collect the results obtained so far and prove Theorems~\ref{t:2}  and  \ref{t:3} stated in the introduction.


Since both $X$ and $\widetilde{X}$ admit continuous transition densities, it follows from Corollary \ref{c:one-implies-all} that if $\widetilde{\Pp}_x\ll \Pp_x$ ($\widetilde{\Pp}_x\perp \Pp_x$) for \emph{some} $x\in \R^d$, then this remains true for \emph{all} $x\in \R^d$. Similarly, if $\Pp_x\ll\widetilde{\Pp}_x$ ($\Pp_x\perp \widetilde{\Pp}_x$) for \emph{some} $x\in \R^d$, then this remains true for \emph{all} $x\in \R^d$.

For the proof of Theorem \ref{t:2} we need that the invariant $\sigma$-field is trivial.
\begin{lemma}\label{l:invariant-trivial}
    The invariant $\sigma$-field $\II$ is trivial with respect to both $\Pp_x$ and $\widetilde{\Pp}_x$ for all $x\in \R^d$.
\end{lemma}
\begin{proof}
    Using standard approximation techniques it is enough to consider non-negative bounded random variables. Suppose that $\Lambda$ is a non-negative and bounded $\II$-measurable random variable. Define $h(x):=\E_x [\Lambda]$. Then $h$ is an invariant function in the sense that $h(x)=\E_x h(X_t)$ for all $t\ge 0$. This implies that $(h(X_t))_{t\geq 0}$ is a bounded martingale and $\Lambda=\lim_{t\to \infty} h(X_t)$ $\Pp_x$-a.s.~for every $x\in \R^d$. By optional stopping, we see that $h(x)=\E_x h(X_{\tau_D})$ for every relatively compact open set $D\subset \R^d$ and every $x\in D$; this means that $h$ is harmonic on $\R^d$.
    Therefore, in order to show that $\II$ is $\Pp_x$-trivial for all $x\in \R^d$, it suffices to prove that constants are the only non-negative bounded harmonic functions with respect to $X$. This is equivalent to showing that the (minimal) Martin boundary of $\R^d$ with respect to $X$ consists of a single point, say $\infty$, cf.\ \cite{KW}.
    For stable L\'evy processes this is a well-known fact which follows from
    \begin{equation}\label{e:green-limit}
        \lim_{|y|\to \infty}\frac{G(x,y)}{G(x_0,y)}=1.
    \end{equation}
    For the purely discontinuous Girsanov transform $\widetilde{X}$ we proceed as follows. First note that because of the conservativeness of $\widetilde{X}$ constant functions are harmonic. It is easy to see that $\widetilde{X}$ satisfies the conditions from \cite{KW}, hence admits the Martin boundary $\partial^M \R^d$. Fix a point $x_0\in \R^d$ and let
    $$
        \widetilde{M}(x,y)=\frac{\widetilde{G}(x,y)}{\widetilde{G}(x_0,y)},
    $$
    where, as before, $\widetilde{G}(x,y)$ denotes the Green function of $\widetilde{X}$. If $z\in \partial^M \R^d$, then the Martin kernel at $z$ is given by
    \begin{equation}\label{e:mk-tildeX}
        \widetilde{M}(x,z)
        =\lim_{y_n\to z}\widetilde{M}(x,y_n)
        =\lim_{y_n\to z} \frac{\widetilde{G}(x,y_n)}{\widetilde{G}(x_0,y_n)}
    \end{equation}
    where $y_n\to z$ in the Martin topology and $|y_n|\to \infty$. Recall from \eqref{e:green-fe} that $c^{-1}G(x,y)\le \widetilde{G}(x,y)\le c \,G(x,y)$ for all $x,y \in \R^d$ where $G$ denotes the Green function of the $\alpha$-stable L\'evy process. We conclude from \eqref{e:green-limit} and \eqref{e:mk-tildeX} that
    \begin{equation}\label{e:mk-estimate}
        c^{-2} \le \widetilde{M}(x,z)\le c^2
        \quad \text{for all\ } x\in \R^d \text{\ and\ }z\in \partial^M \R^d.
    \end{equation}
    Suppose that $h$ is a minimal harmonic function with respect to $\widetilde{X}$. It follows from \cite{KW} that there is a finite measure $\nu$ on $\partial^M \R^d$ such that
    $$
        h(x)=\int_{\partial^M \R^d} \widetilde{M}(x,z)\, \nu(dz).
    $$
    From \eqref{e:mk-estimate} we conclude that $c^{-2}\nu(\partial^M \R^d)\le h$. The minimality of $h$ now implies that it is constant. Thus, the minimal Martin boundary of $\R^d$ with respect to $\widetilde{X}$ consists of a single point.
\end{proof}

\begin{proof}[Proof of Theorem \ref{t:2}]
    \ref{t:2-a} follows immediately from Corollary \ref{c:dichotomy} and Lemma \ref{l:invariant-trivial}.

    \ref{t:2-b} The assumption $\widetilde{\Pp}_x\sim \Pp_x$ implies, by Theorem \ref{t:1}, that
    $$
        \Pp_x\left(\sum_{t>0}F^2(X_{t-}, X_t)<\infty\right)
        = 1
    $$
    Note that
    $$
    F^2(x,y)\le C^2 \frac{|x-y|^{2\beta}}{(1+|x|^{\beta}+|y|^{\beta})^2}\le C^2 \frac{|x-y|^{2\beta}}{1+|x|^{2\beta}+|y|^{2\beta}}\,.
    $$
    Since $2\beta  > \alpha$,
    Theorem \ref{t:finite-expectation} shows that $\sup_{x\in \R^d} \E_x \left[\sum_{t>0} F^2 (X_{t-}, X_t)\right] <\infty$. From Theorem \ref{t:1}  (\~c)  we conclude that $\sup_{x\in \R^d} \entropy(\Pp_x; \widetilde{\Pp}_x)<\infty$.
\end{proof}

\begin{proof}[Proof of Theorem \ref{t:3}]
    By assumption $2\gamma <\alpha <2\beta$. Denote by $B(x_n,r_n)$ the balls and by $\Phi$ the function constructed in the proof of Theorem \ref{t:infinite-expectation} with $2\gamma$ and $2\beta$. Thus
    $$
        \Phi(x,y)
        =
        \begin{cases}
            \dfrac{|x-y|^{2\beta}}{|x|^{2\gamma}+|y|^{2\gamma}}, & x,y\in B(x_n,r_n) \text{\ for some $n$ and $|x-y|<1$},\\[\bigskipamount]
            0, & \text{otherwise}.
        \end{cases}
    $$
    Define $F(x,y)=\frac18 \sqrt{\Phi(x,y)}$. Then
    $$
        F(x,y)
        \leq \frac18\, \frac{|x-y|^{\beta}}{\sqrt{|x|^{2\gamma}+|y|^{2\gamma}}}
        \le \frac14 \frac{|x-y|^{\beta}}{|x|^{\gamma}+|y|^{\gamma}}
        \le \frac12 \frac{|x-y|^{\beta}}{1+|x|^{\gamma}+|y|^{\gamma}}
    $$
    since we can take $|x|,|y|\ge 1$.
    As $\sum_{t>0}F^2(X_{t-}, X_t)=\frac18 \sum_{t>0}\Phi(X_{t-},X_t)$, we see from Theorem \ref{t:infinite-expectation} that $\sum_{t>0}F^2(X_{t-}, X_t)<\infty$ $\Pp_x$ a.s.~and $\E_x[\sum_{t>0}F^2(X_{t-}, X_t)]=\infty$. By Theorem \ref{t:1}(\~b) and (\~c) it follows that $\Pp_x\ll \widetilde{\Pp}_x$ and $\entropy(\Pp_x;\widetilde{\Pp}_x)=\infty$.
\end{proof}

\noindent
\textsc{Ren\'e L.~Schilling:} Institut f\"ur Mathematische Stochastik, Fachrichtung Mathematik, Technische Universit\"at Dresden, 01062 Dresden, Germany.\\ \emph{email:} \texttt{rene.schilling@tu-dresden.de}

\bigskip
\noindent
\textsc{Zoran Vondra\v{c}ek:} Department of Mathematics, Faculty of Science, University of Zagreb, Bijeni\v{c}ka c.~30, 10000 Zagreb, Croatia. \\ \emph{email:} \texttt{vondra@math.hr}
\end{document}